\documentclass[12pt,reqno]{amsart}
  \usepackage{latexsym} 
  \usepackage[all]{xy}
  \usepackage{amsfonts} 
  \usepackage{amsthm} 
  \usepackage{amsmath} 
  \usepackage{amssymb}
  \usepackage{pifont}  
  \usepackage{enumerate}
  \xyoption{2cell}

 \usepackage{tikz}
\usetikzlibrary{arrows}

  \def\<{{\langle}} 
  \def\>{{\rangle}}

  \def\note#1{{}}

  \def\note#1{} 

  \def\hom#1#2#3{{{\rm Hom}\sb{#1}(#2,#3)}}

  \def\beq{\begin{equation}} 
  \def\eeq{\end{equation}}

  \def\id{\mathrm{id}}

  \def\ot{{\otimes}}

  \newcounter{zlist} 
  \newenvironment{zlist}{\begin{list}{(\arabic{zlist})}{ 
  \usecounter{zlist}\leftmargin2.5em\labelwidth2em\labelsep0.5em 
  \topsep0.6ex%\itemsep0.3ex plus0.2ex minus0.3ex 
  \parsep0.3ex plus0.2ex minus0.1ex}}{\end{list}}

  \newcounter{blist} 
  \newenvironment{blist}{\begin{list}{(\alph{blist})}{ 
  \usecounter{blist}\leftmargin2.5em\labelwidth2em\labelsep0.5em 
  \topsep0.6ex %\itemsep0.3ex plus0.2ex minus0.3ex 
  \parsep0.3ex plus0.2ex minus0.1ex}}{\end{list}} 

  \newcounter{rlist} 
  \newenvironment{rlist}{\begin{list}{(\roman{rlist})}{ 
  \usecounter{rlist}\leftmargin2.5em\labelwidth2em\labelsep0.5em 
  \topsep0.6ex %\itemsep0.3ex plus0.2ex minus0.3ex 
  \parsep0.3ex plus0.2ex minus0.1ex}}{\end{list}}

\def\stac#1{\raise-.2cm\hbox{$\stackrel{\displaystyle\otimes}{\scriptscriptstyle{#1}}$}}

\def\cten#1{\raise-.2cm\hbox{$\stackrel{\displaystyle\widehat{\otimes}}
{\scriptscriptstyle{#1}}$}}

  \def\Label#1{\label{#1}\ifmmode\llap{[#1] }\else 
  \marginpar{\smash{\hbox{\tiny [#1]}}}\fi} 
  \def\Label{\label}

  \newtheorem{proposition}{Proposition}[section]
  \newtheorem{lemma}[proposition]{Lemma} 
  \newtheorem{corollary}[proposition]{Corollary} 
  \newtheorem{theorem}[proposition]{Theorem} 

\theoremstyle{definition} 
  \newtheorem{definition}[proposition]{Definition}
    
  \newtheorem{example}[proposition]{Example}

  \theoremstyle{remark} 
  \newtheorem{remark}[proposition]{Remark}

  \newcounter{c} 
   
  \newcommand{\etyk}[1]{\vspace{-7.4mm}$$\begin{equation}\Label{#1} 
  \addtocounter{c}{1}} 
  \renewcommand{\]}{\ifnum \value{c}=1 $$\else \end{equation}\fi} 
\def\ot{\otimes}

\def\CC{{\mathbb C}}

\def\KK{{\mathbb K}}

\def\NN{{\mathbb N}}

\def\ZZ{{\mathbb Z}}

\def\bbb{{\mathfrak b}}

\newcommand{\Cc}{\mathcal{C}}

\def\*C{{}^*\hspace*{-1pt}{\Cc}}

\def\text#1{{\rm {\rm #1}}}

 \def\1{\mathbf{1}}

 \def\Der#1#2{\mathrm{Der} _{#1}(#2)}

\begin{document}

\title{Skew derivations on generalized Weyl algebras}

\author{Munerah Almulhem}
\address{ Department of Mathematics, Swansea University, 
  Swansea SA2 8PP, U.K.} 
  \email{844404@swansea.ac.uk}   

\author{Tomasz Brzezi\'nski}
 \address{ Department of Mathematics, Swansea University, 
  Swansea SA2 8PP, U.K.\ \newline 
\indent Department of Mathematics, University of Bia{\l}ystok, K.\ Cio{\l}kowskiego  1M,
15-245 Bia\-{\l}ys\-tok, Poland} 
  \email{T.Brzezinski@swansea.ac.uk}   
 
\begin{abstract}
A wide class of skew derivations  on degree-one generalized Weyl algebras $R(a,\varphi)$ over a ring $R$ is constructed. All these derivations are twisted by a degree-counting extensions of automorphisms of $R$. It is determined which of the constructed derivations are $Q$-skew derivations. The compatibility of these skew derivations with the natural $\ZZ$-grading of $R(a,\varphi)$ is studied. Additional classes of skew derivations are constructed for generalized Weyl algebras given by an automorphism $\varphi$ of a finite order. Conditions that the central element $a$ that forms part of the structure of $R(a,\varphi)$ need to satisfy for the orthogonality of pairs of aforementioned skew derivations are derived.  General constructions are illustrated by classification of skew derivations of generalized Weyl algebras over the polynomial ring in one variable and with a linear polynomial as the central element. 
\end{abstract}
\maketitle

\section{Introduction}
This paper is devoted to the construction of a class of skew derivations of  {\em degree-one generalized Weyl algebras}. These algebras arose almost simultanously in ring theory \cite{Bav:gen}, \cite{Bav:ten}, \cite{Jor:pri} and non-commutative algebraic geometry  \cite{Ros:alg}, \cite{LunRos:Kas} (where they are called {\em rank-one hyperbolic algebras}), and they have become a subject of intensive study motivated in particular by the fact that many examples of algebras arising from quantum group theory or non-commutative geometry fall into this class. Generalized Weyl algebras $R(a,\varphi)$  are obtained as extensions of a ring $R$ by adjoining two additional generators that satisfy relations determined by an automorphism $\varphi$ of $R$ and an element $a$ in the centre of $R$ (see Section~\ref{sec.prel} for the precise definition), and they can be understood as generalizations of skew Laurent polynomial rings. 

The motivation for this study, results of which are being presented to the reader herewith, comes from non-commutative differential geometry, where skew derivations often play the role of vector fields (cf.\ \cite[Section~4.4]{Mad:int}) and may be used to equip non-commutative spaces with differential structures. As the Leibniz rule for skew derivations studied here is twisted by an automorphism, one first should make a choice of a suitable automorphism. Automorphism groups of generalized Weyl algebras have been studied in special cases, for example in the case of quantum generalized Weyl algebras \cite{BavJor:iso}, \cite{RicSol:iso}, \cite{SuaViv:aut} or generalized down-up algebras \cite{CarLop:aut}, to mention but a few. Our aim, however, is to work in a general degree-one situation, and hence we construct skew derivations twisted by automorphisms that can be defined for 
any generalized Weyl algebra over $R$. Such automorphisms are determined by an automorphism $\sigma$ of $R$ compatible with the data defining the generalized Weyl algebra, and a central unit $\mu$ in $R$ (see Lemma~\ref{lemma.auto} for details). We term them {\em degree-counting extensions of $\sigma$} of {\em coarseness} $\mu$. 

In the main Section~\ref{sec.skew} of the present paper we construct a wide class of skew derivations  (twisted by degree-counting extensions of $\sigma \in \mathrm{Aut}(R)$) on degree-one generalized Weyl algebras $R(a,\varphi)$. Each element in this class is determined by the datum comprising a system of skew derivations of $R$ and a pair of elements of $R$, all of which are required to satisfy a set of natural conditions (see Theorem~\ref{thm.Weyl.deriv}). Individually, each assignment of a skew derivation on $R(a,\varphi)$ to a skew derivation on $R$ defines an injective map of twisted degree-one Hochschild cohomology groups. We show also that our construction affords one a full classification of skew derivations which send $R$ to a positive (respectively, negative) part of $R(a,\varphi)$ (the positivity or negative is defined with respect to a natural $\ZZ$-grading) and vanish on one of the extending generators of $R(a,\varphi)$. Next we determine which of the constructed skew derivations are $Q$-skew derivations and we also derive the sufficient and necessary conditions for compatibility of skew derivations with the natural $\ZZ$-grading of $R(\varphi; a)$ as maps of a fixed degree. Departing from the general case, we focus on algebras associated to automorphisms $\varphi$ of finite order, and  construct additional classes of skew derivations on them. 

Keeping in mind that skew derivations can be used to construct first-order differential calculi, provided they satisfy particular {\em orthogonality conditions} (see Section~\ref{sec.prel} for explanation), in Section~\ref{sec.ortho} we derive 
sufficient conditions for the orthogonality of pairs of skew derivations constructed in Theorem~\ref{thm.Weyl.deriv}. The bulk of these conditions involves the pairwise co-primeness  of $a$  with $\varphi^i(a)$, which incidentally is crucial for the statement of the Kashiwara theorem for generalized Weyl algebras \cite[2.2~Theorem]{LunRos:Kas}. In this way some of the results of \cite{Brz:dif}, where orthogonal systems of skew derivations were studied for generalized Weyl algebras over a polynomial ring in one variable, can be reproduced as special cases of a far more general situation.

In the final Section~\ref{sec.disc} we focus on generalized Weyl algebras over the polynomial ring in one indeterminate $h$ with coefficients from a field $\KK$, and with a linear polynomial as the central element. The automorphism $\varphi$ is chosen to be the map rescaling $h$ by a non-zero scalar $q\in \KK$. We classify all skew derivations twisted by a degree-counting extension of the identity automorphism of coarseness equal to $q$ and construct orthogonal pairs of skew derivations on the quantum disc algebra, i.e.\ the generalized Weyl algebra over $\KK[h]$ given by the central element $a=1-h$ and the automorphism $h\mapsto qh$.

\section{Preliminaries}\label{sec.prel}\setcounter{equation}{0}

Given an associative, unital ring $R$, a ring automorphism $\varphi: R\to R$ and an element $a$ of the centre of $R$, the associated {\em degree-one generalized Weyl algebra} $R(a,\varphi)$ is defined as the quotient of the free polynomial ring  $R\langle x, y\rangle$ by the relations:
 \begin{equation}\label{gW}
 xy = \varphi(a), \quad yx = a, \quad xr = \varphi(r)x, \quad yr = \varphi^{-1}(r)y, 
 \end{equation}
for all $r\in R$. Every element of $R(a,\varphi)$ can be uniquely written as $r+ \sum_{k>0} r_k x^k + \sum_{l>0} s_l y^l$, where $r, r_k,s_l\in R$. In the sequel, by a generalized Weyl algebra we always mean a degree-one generalized Weyl algebra. 

If $R$ is a $\ZZ$-graded algebra, then $R(a,\varphi)$ can also be made into a $\ZZ$-graded algebra provided that $\varphi$ is a degree-preserving automorphism and $a$ is a homogenous element. Specifically, if the degree of $a$ is $d$, then $x$ can be set to have, say, a positive degree $m$ and $y$ to have degree $d-m$. We refer to this grading of $R(a,\varphi)$ as the {\em $(d,m)$-type grading}. In particular if $R$ is concentrated in the degree zero (or, simply, not graded), then we set 
$$
R(a,\varphi)_0 = R, \quad \! R(a,\varphi)_+ = \{\sum_{m=1} r_m x^m \; |\; r_m\in R\}, \quad \! R(a,\varphi)_- = \{\sum_{m=1} r_m y^m \; |\; r_m\in R\},
$$
so that
$$
R(a,\varphi) = R(a,\varphi)_- \oplus  R(a,\varphi)_0 \oplus R(a,\varphi)_+. 
$$
We refer to $R(a,\varphi)_+$ (respectively, $R(a,\varphi)_-$) as to the {\em positive} (respectively, {\em negative}) part of $R(a,\varphi)$.

Although the definition of $R(a,\varphi)$ is not invariant under the exchange of generators $x$ and $y$, one easily checks that the following map
\begin{equation}\label{x-y}
\Psi: R(a,\varphi)\to R(\varphi(a),\varphi^{-1}), \quad x\mapsto y, \quad y\mapsto x,\quad \Psi|_R = \id_R,
\end{equation}
(where we denote the generators of two generalized Weyl algebras by the same letters) is an isomorphism of algebras; see \cite[2.7~Lemma~(i)]{BavJor:iso}. We refer to this isomorphism as to the {\em $x$-$y$ symmetry} (it is called a {\em Fourier transform} in \cite{LunRos:Kas}). This symmetry allows one to deduce counterparts of various statements, without any additional effort.

For any ring $A$, a {\em (right) skew derivation} is a pair $(\partial, \sigma)$ consisting of a ring automorphism $\sigma: A\to A$ and an additive map $\partial: A\to A$ that satisfies the $\sigma$-twisted Leibniz rule, for all $a,b\in A$,
\begin{equation}\label{Leibniz}
\partial(ab) = \partial(a)\sigma(b) + a\partial(b).
\end{equation} 
Clearly, if $(\partial_1, \sigma)$ and $(\partial_2, \sigma)$ are skew derivations, then so  are $(\partial_1+\partial_2, \sigma)$ and $(-\partial_i,\sigma)$. Obviously, $(0,\sigma)$ is a skew derivation. Hence the set $\Der\sigma A$ of all skew derivations $(\partial,\sigma)$ of $A$ with a fixed $\sigma$ is an abelian group. 

To any element $b\in A$ one can associate the corresponding {\em inner} skew derivation $(\partial_b, \sigma)$ given by the $\sigma$-twisted commutator with $b$, i.e., for all $a\in A$,
$$
\partial_b(a) = b\sigma(a) - ab.
$$
The assignment $b\mapsto (\partial_b,\sigma)$  defines an additive map 
\begin{equation}\label{Delta}
\Delta: A\to \Der\sigma A.
\end{equation}

A skew derivation $(\partial, \sigma)$ is called a {\em skew $Q$-derivation}, if there exists a central unit $Q \in A$, invariant under $\sigma$ and such that
\begin{equation}\label{q-derivation}
\sigma\circ\partial\circ\sigma^{-1} = Q\, \partial.
\end{equation}
Any inner skew derivation $(\partial_b,\sigma)$ is a skew 1-derivation.

Given a ring automorphism $\sigma$ of $A$ and an $A$-bimodule $M$, we write $M_\sigma$ for the $A$-bimodule with right $A$-action twisted by $\sigma$, i.e.\ defined by
$$
m\cdot a := m\sigma(a), \qquad \mbox{for all $a\in A$, $m\in M$}.
$$
With this notation a pair $(\partial, \sigma)$ is a skew derivation on $A$ if and only if $\partial$ is an $A_\sigma$-valued derivation of $A$. Again, for an $A$-bimodule $M$ we denote by $M^A$ the abelian group
$$
M^A:= \{ m\in M\; |\; \forall a\in A,\, am=ma\}.
$$
Obviously $M^A$ is a module over the centre $Z(A) = A^A$ of $A$.

Recall that, for a ring $A$ and an $A$-bimodule $M$, the $M$-valued Hochschild cohomology of $A$, $HH(A,M)$, is the cohomology of the complex
$\bbb : HC(A,M)^n \to HC(A,M)^{n+1},
$
where
$
HC(A,M)^n = \hom {} {A^{\ot n}} M,
$
the group of additive homomorphisms from $A^{\ot n}$ to $M$, and
\begin{eqnarray}\label{hoch}
(\bbb f)(a_0 \ot \cdots \ot a_{n}) = a_0 f(a_1 \ot \cdots \ot a_n) +&&\!\!\!\!\!\!\!\! \sum_{k=1}^n (-1)^nf(a_0\ot \cdots \ot a_{k-1}a_k \ot \cdots \ot a_n) \nonumber \\
&& + (-1)^{n+1} f(a_0 \ot \cdots \ot a_{n-1})a_n.
\end{eqnarray}
In particular, $HH^0(A,M) = M^A$ and $HC^1(A,M)$ is the group of $M$-valued derivations on $A$, while the image of $\bbb : HC^0(A,M) \to HC^1(A,M)$ consists of all inner derivations. Thus, the kernel of the map $\Delta$ \eqref{Delta} is simply equal to $HH^1(A,A_\sigma)$.

The complex $(HC(A,A_\sigma), \bbb)$, where $\sigma$ is an algebra automorphism of $A$, contains a subcomplex, which will play a special role in the discussion of skew derivations on generalized Weyl algebras. Let $\varphi$ be an automorphism of $A$ commuting with $\sigma$, and let $\mu$ be an element of the centre of $A$. Set
$$
HC_{\sigma;\mu,\varphi}^n(A) := \{ f\in \hom {}{A^n} A \; |\; \varphi^{-1}\circ f\circ \varphi^{\ot n} = \mu\, f\} .
$$
Then the Hochschild coboundary $\bbb$ \eqref{hoch} for $M = A_\sigma$ restricts to $HC_{\sigma;\mu,\varphi}(A)$. The cohomology of the resulting complex is denoted by $HH_{\sigma;\, \mu,\varphi}(A)$, and we refer to it as a {\em doubly twisted Hochschild cohomology of $A$}. In case $\mu =1$, $\varphi =\id$ this is the standard twisted Hochschild cohomology of $A$, denoted by $HH_{\sigma}(A)$

Let $(\partial_i, \sigma_i)_{i=1}^n$ be a (finite) family of skew derivations of a ring $A$. We say that it forms an {\em orthogonal system of skew derivations} provided there exist two finite sets $\{a_{i\,t}\}, \{b_{i\,t}\}\subset A$  such that, 
\begin{equation}\label{ortho}
\sum_t a_{i\,t}\partial_k(b_{i\,t}) = \delta_{ik}, \quad \mbox{for all}\quad  i,k =1, \ldots, n.
\end{equation}
Note that this is equivalent to to the existence of  three finite sets $\{a_{i\,t}\}, \{b_{i\,t}\}, \{c_{i\,t}\}$ of elements of $A$ such that
\begin{equation}\label{ortho.x}
\sum_t a_{i\,t}\partial_k(b_{i\,t}) \sigma_k\left(\sigma_i^{-1}\left(c_{i\,t}\right)\right) = \delta_{ik}, \quad \mbox{for all}\quad  i,k =1, \ldots, n.
\end{equation}
Indeed, obviously \eqref{ortho.x} implies \eqref{ortho}. On the other hand, if \eqref{ortho.x} holds, then the twisted Leibniz rules yield
$$
\sum_t  a_{i\,t}\partial_k\left(b_{i\,t}\sigma_i^{-1}\left(c_{i\,t}\right)\right) - \sum_t  a_{i\,t}b_{i\,t}\partial_k\left(\sigma_i^{-1}\left(c_{i\,t}\right)\right) = \sum_t a_{i\,t}\partial_k(b_{i\,t}) \sigma_k\left(\sigma_i^{-1}\left(c_{i\,t}\right)\right) = \delta_{ik},
$$
hence $\{ a_{i\,t}, -a_{i\,t}b_{i\,t}\}, \{b_{i\,t}\sigma_i^{-1}\left(c_{i\,t}\right),\sigma_i^{-1}\left(c_{i\,t}\right)\}$ are the required two sets.

As explained for example in \cite[Section~3]{BrzElK:int} orthogonal systems of skew derivations on $A$ can be used to form first order differential calculi on $A$. By the latter we mean a pair consisting of an $A$-bimodule $\Omega$ and an $\Omega$-valued derivation $d: A\to \Omega$, such that $\Omega = Ad(A)$. Given an orthogonal system of skew derivations $(\partial_i, \sigma_i)_{i=1}^n$, $\Omega$ and $d$ are defined by,
\begin{equation}\label{1oc}
\Omega = \bigoplus_{i =1}^n A_{\sigma_i}, \qquad d: a\mapsto \left(\partial_i(a)\right)_{i=1}^n.
\end{equation}
While the $\sigma$-twisted Leibniz rule ensures that the map $d$ in \eqref{1oc} is a derivation, the orthogonality conditions \eqref{ortho} are equivalent to the density of $\Omega$: $\Omega = Ad(A)$. 

In this paper we investigate skew derivations of generalized Weyl algebras $R(a,\varphi)$ related to a particular class of automorphisms of $R(a,\varphi)$ (compare \cite[2.7~Lemma~(iii)]{BavJor:iso}).
\begin{lemma}\label{lemma.auto}
Given $R(a,\varphi)$, let $\sigma$ be a ring automorphism of $R$ such that 
\begin{equation}\label{sigma.phi}
\sigma \circ \varphi = \varphi\circ\sigma, \qquad \sigma(a) =a.
\end{equation}
Then, for any central unit $\mu$ in $R$, the map $\sigma$ extends to the automorphism $\sigma_\mu$ of $R(a,\varphi)$ by
\begin{equation}\label{sigma.mu}
\sigma_\mu(x) = \mu^{-1} x , \qquad \sigma_\mu(y) =y\,\mu = \varphi^{-1}(\mu) y.
\end{equation}
\end{lemma}
\begin{proof}
One easily verifies that $\sigma_\mu$ is compatible with relations \eqref{gW}.
\end{proof}

Thinking about $R(a,\varphi)$ as a $\ZZ$-graded algebra we feel justified in making the following
\begin{definition}\label{def.degree}
An automorphism $\sigma_\mu$ described in Lemma~\ref{lemma.auto} is called a {\em degree-counting extension} of the automorphism $\sigma$ of $R$ (of {\em coarseness} $\mu$).
\end{definition}

\section{Skew derivations on generalized Weyl algebras}\label{sec.skew}\setcounter{equation}{0}
In this section first we describe a wide class of skew derivations on a generalized Weyl-algebra $R(a,\varphi)$, twisted by the degree-counting extension of an automorphism $\sigma$ of $R$ of a general coarseness $\mu$. Next we determine which of the constructed derivations are $Q$-skew derivations. Finally, we construct additional skew derivations, when the automorphism $\varphi$ has a finite order.

\begin{theorem}\label{thm.Weyl.deriv}
Let $R(a,\varphi)$ be a generalized Weyl algebra and let $\sigma$ be an automorphism of $R$ commuting with $\varphi$ and fixing $a$. Let $\sigma_\mu$ be the degree-counting extension of $\sigma$ of coarseness $\mu$. Let
$$
(\alpha_i ,\ \varphi^i\circ \sigma)_{i=-N}^M, 
$$
be skew derivations on $R$ such that, for all $i=-N,\ldots , M$,
\begin{equation}\label{q-deriv}
\alpha_i\circ \varphi= \varphi^{i}(\mu) \varphi \circ \alpha_i  , 
\qquad a\mid \alpha_0(a), \qquad \frac{\alpha_0(a)}a \in R_\sigma^R.
\end{equation}
For all $c\in R_\sigma^R$, $b\in R$, set 
\begin{subequations}\label{delta}
\begin{equation}\label{delta.r}
\partial(r) = \sum_{m=0}^M\alpha_m(r)\, x^m +\sum_{n=1}^N\alpha_{-n}(r)\,  y^n +b\sigma(r) - rb, \qquad \mbox{for all $r\in R$},
\end{equation}
\begin{equation}\label{delta.x}
\partial(x) = \left(c - \varphi(b) + \mu^{-1} b\right) x + \sum_{n=1}^N \varphi \left(\alpha_{-n}\left(a\right)\right) y^{n-1},
\end{equation}
\begin{eqnarray}\label{delta.y}
\partial(y) &=& \left(\frac{\alpha_0(a)}a - \varphi^{-1}\left( c + \mu^{-1} b\right) +b \right) y\, \mu  + \sum_{m=1}^M \alpha_m(a)\, x^{m-1} \mu  \\
&=& \left(\frac{\alpha_0(a)}a - \varphi^{-1}\left( c + \mu^{-1} b\right) +b \right)\varphi^{-1}(\mu) y  + \sum_{m=1}^M \alpha_m(a)\, \varphi^{m-1}(\mu)x^{m-1}. \nonumber
\end{eqnarray}
\end{subequations}
Then $\partial$ extends to a skew derivation $(\partial, \sigma_\mu)$ on $R(a,\varphi)$.
\end{theorem}
\begin{proof}
Since, for a fixed automorphism, skew derivations form an abelian group, to prove the theorem we only need to check that the following four (classes of) maps defined by 
\begin{subequations}\label{monomials}
\begin{equation}\label{monomial.0}
\partial_0(r) = \alpha_0 (r), \qquad \partial_0(x) = cx, \qquad \partial_0(y) =  \left(\frac{\alpha_0(a)}a - \varphi^{-1}\left( c\right)\right) y\,  \mu,
\end{equation}
\begin{eqnarray}\label{monomial.m}
\partial_m(r) = \alpha_m (r) x^m, \quad \partial_m(x) = 0, \quad \partial_m(y) \!\!\!&=& \!\!\!  \alpha_m(a)\, x^{m-1}\mu \nonumber\\
&=& \!\!\! \varphi^{-1}\left(\alpha_m\left(\varphi\left(a\right)\right)\right) x^{m-1},
\end{eqnarray}
\begin{equation}\label{monomial.-n}
\partial_{-n}(r) = \alpha_{-n} (r) y^n, \qquad \partial_{-n}(x) = \varphi\left(\alpha_{-n}\left(a\right)\right) y^{n-1}, \qquad \partial_{-n}(y) = 0,
\end{equation}
\begin{eqnarray}\label{monomial.00}
\bar\partial_b(r) = b\sigma(r) - rb, \quad \bar\partial_b(x) =  \left(\mu^{-1} b - \varphi(b) \right) x, \quad \bar\partial_b(y) \!\!\!&=& \!\!\! \varphi^{-1} \left( \mu\,\varphi\left( b\right) - b\right)  y \nonumber\\
& = &\!\!\! y\, \left( \mu\,\varphi\left( b\right) - b\right),
\end{eqnarray}
\end{subequations}
 where $m,n$ are positive integers, extend to the elements of $\Der{\sigma_\mu}{R(a,\varphi)}$. The second equality in the definition of $\partial_m(y)$ \eqref{monomial.m} follows by the commutation rules in a generalized Weyl algebra and by  \eqref{q-deriv}.

Since $(\alpha_0,\sigma)$ is a skew derivation on $R$, and $\sigma_\mu$ restricted to $R$ is equal to $\sigma$, $\partial_0$ restricted to $R$ satisfies the $\sigma_\mu$-twisted Leibniz rule.  We need to check that $\partial_0$ extended to the whole of $R(a,\varphi)$ by the $\sigma_\mu$-twisted Leibniz rule preserves all the relations \eqref{gW}. First, let us compute
\begin{eqnarray*}
\partial_0(xy - \varphi(a)) &=&  c x  y\mu + x  \left(\frac{\alpha_0(a)}a - \varphi^{-1}\left( c\right)\right) y \mu- \alpha_0\left(\varphi(a)\right) \\
&=& \mu c \varphi(a) + \mu \left(\varphi\left(\frac{\alpha_0(a)}a\right) - c\right) \varphi(a) - \mu \varphi(\alpha_0(a)) =0,
\end{eqnarray*}
where the first equality follows by the definition of $\partial_0$ and the $\sigma_\mu$-twisted Leibniz rule, while the second one follows by relations \eqref{gW}, the first of \eqref{q-deriv}, and the centrality of $\mu$. 
Next:
\begin{eqnarray*}
\partial_0(yx - a) &=& \left(\frac{\alpha_0(a)}a - \varphi^{-1}\left( c\right)\right) y \mu \mu^{-1} x + y c x  - \alpha_0(a) \\ 
&=&  \left(\frac{\alpha_0(a)}a - \varphi^{-1}\left( c\right)\right) a + \varphi^{-1}(c) a  - \alpha_0(a) =0,
\end{eqnarray*}
where the first equality follows by the definition of $\partial_0$ via the twisted Leibniz rule, and the second one by \eqref{gW} and the centrality of $\mu$. Next, for all $r\in R$,
\begin{eqnarray*}
\partial_0(xr - \varphi(r)x) &=& cx \sigma(r) + x\alpha_0(r) - \alpha_0 (\varphi(r)) \mu^{-1}  x - \varphi(r) c x \\
&=& \left(c\varphi( \sigma(r))  + \varphi(\alpha_0(r)) -  \varphi(\alpha_0 (r)) - c \sigma(\varphi(r))\right) x =0.
\end{eqnarray*}
Here, as before, the first equality follows by the definition of $\partial_0$ and the $\sigma_\mu$-twisted Leibniz rule, the second one follows by the first of conditions \eqref{q-deriv}, centrality of $\mu$ and by the fact that $c\in R_\sigma^R$. The final equality is a consequence of \eqref{sigma.phi}. Finally, for all $r\in R$,
\begin{eqnarray*}
\partial_0(yr - \varphi^{-1}(r)y) &=& \left(\frac{\alpha_0(a)}a - \varphi^{-1}\left( c\right)\right) y  \mu \sigma(r) + y\alpha_0(r)\\
&&  - \alpha_0 (\varphi^{-1}(r))   y \mu - \varphi^{-1}(r) \left(\frac{\alpha_0(a)}a - \varphi^{-1}\left( c\right)\right) y \mu  \\
&=& \left(\frac{\alpha_0(a)}a - \varphi^{-1}\left( c\right)\right)  \varphi^{-1}(\sigma(r))y  \mu + \varphi^{-1}(\alpha_0(r)) y \\
&&  - \varphi^{-1}(\alpha_0(r)) y -  \varphi^{-1}(r) \left(\frac{\alpha_0(a)}a - \varphi^{-1}\left( c\right)\right) y \mu \\
&=&  \frac{\alpha_0(a)}a\sigma\left(\varphi^{-1}(r)\right)y\mu - \varphi^{-1}(r) \frac{\alpha_0(a)}a y\mu  =0.
\end{eqnarray*}
Again, the first equality follows by the definition of $\partial_0$ and the $\sigma_\mu$-twisted Leibniz rule. The second equality is a consequence of \eqref{gW}, the first of conditions \eqref{q-deriv} and the centrality of $\mu$, while the third one follows by \eqref{sigma.phi} and the fact that $c\in R_\sigma^R$. Since also $\frac{\alpha_0(a)}a \in R_\sigma^R$, the final equality is obtained. Thus, $\partial_0$ vanishes on all generators of the ideal in $R\langle x,y\rangle$ that defines $R(a,\varphi)$, hence $\partial_0$ extends as a $\sigma_\mu$-twisted derivation to the whole of $R(a,\varphi)$.

To prove that $(\partial_m, \sigma_\mu)$ is a skew derivation we set,
$$
\partial_m(r) = \alpha_m (r) x^m, \qquad \partial_m(x) = 0, \qquad \partial_m(y) =  c_m x^{m-1},
$$
where $\alpha_m$ is an additive automorphism of $R$ and $c_m\in R$, and derive what conditions these need to satisfy. First, for all $r,s\in R$,
$$
\partial_m(rs) = \partial_m(r)\sigma(s) + r\partial_m(s),
$$
hence, by \eqref{gW}, 
$$
\alpha_m(rs) x^m = \alpha_m(r) x^m \sigma(s) + r\alpha_m(s)x^m = \left(  \alpha_m(r)\varphi^m(\sigma(s)) + r\alpha_m(s)\right)x^m,
$$
which is equivalent to the fact that $(\alpha_m, \varphi^m\circ \sigma)$ is a skew derivation of $R$. Next, for all $r\in R$,
$$
\partial_m(xr - \varphi(r)x) = x\alpha_m(r)x^m - \alpha_m(\varphi(r)) x^m \mu^{-1}\, x = \left(\varphi(\alpha_m (r)) - \varphi^m(\mu^{-1} ) \alpha_m(\varphi(r))\right)x^{m+1},
$$
by \eqref{gW} and the centrality of $\mu$. This yields necessarily the first of conditions \eqref{q-deriv}.
 Furthermore,
$$
\partial_m(yx - a) = c_m x^{m-1} \mu^{-1}\, x - \alpha_m(a) x^m = \left( \varphi^{m-1}\left(\mu^{-1}\right) c_m - \alpha_m(a)\right)x^m,
$$
by \eqref{gW} and the centrality of $\mu$ . This fixes $c_m$,
\begin{equation}\label{fix.c_m}
c_m = \varphi^{m-1}\left(\mu\right)\alpha_m(a), \quad {\mbox i.e.}\quad \partial_m(y) =  \alpha_m(a)\, x^{m-1}\mu .
\end{equation}
With this at hand we can compute, for all $r\in R$,
\begin{eqnarray*}
\partial_m(yr)&=&\alpha_m(a)\, x^{m-1}\mu \sigma(r) + y \alpha_m(r) x^m\\
&=&\alpha_m(a) \varphi^{m-1}(\sigma(r)) x^{m-1} \mu + \varphi^{-1}(\alpha_m(r)) a x^{m-1}\\
&=& \alpha_m(a) \varphi^{m-1}(\sigma(r)) x^{m-1} \mu + \varphi^{m-1}(\mu) \alpha_m(\varphi^{-1}(r)) ax^{m-1}\\
&=&\left( \alpha_m(a) \varphi^m\left( \sigma( \varphi^{-1}(r)))+a \alpha_m(\varphi^{-1}(r)\right) \right)x^{m-1}\mu
=\alpha_m\left(a\varphi^{-1}(r)\right)x^{m-1}\mu,
\end{eqnarray*}
where the first equality follows by the definition of $\partial_m $ through equation \eqref{fix.c_m}. The second and the third equalities follow by \eqref{gW},  the centrality of $\mu$ and \eqref{q-deriv}, while the third one holds since $\alpha_m$ is a $\varphi^m\circ\sigma$-skew derivation. On the other hand, using \eqref{fix.c_m}, \eqref{gW}, and that $\sigma$ fixes central $a$, and that $\alpha_m$ is a $\varphi^m\circ\sigma$-skew derivation we compute
\begin{eqnarray*}
\partial_m\left(\varphi^{-1}(r)y\right)&=&\alpha_m(\varphi^{-1}(r)) x^m  y\, \mu + \varphi^{-1}(r) \alpha_m(a)\, x^{m-1}\mu \\
&=&\left(\alpha_m(\varphi^{-1}(r)) \varphi^{m}(a) x^{m-1} + \varphi^{-1}(r) \alpha_m(a) x^{m-1}\right)\mu\\
& = &\alpha_m\left(a\varphi^{-1}(r)\right)x^{m-1}\mu.
\end{eqnarray*}
Therefore $\partial_m(yr - \varphi^{-1}(r)y) =0$, as required.
Finally,
\begin{eqnarray*}
\partial_m(xy - \varphi(a))&=& x \alpha_m(a)\, x^{m-1}\mu - \alpha_m(\varphi(a)) x^m\\
&=&  \left( \varphi^{m}(\mu)\varphi(\alpha_m(a)) - \alpha_m(\varphi(a))\right)x^m =0,
\end{eqnarray*}
by \eqref{fix.c_m}, \eqref{gW} and \eqref{q-deriv}.
Thus, the $\sigma_{\mu}$-skew derivation property of $\partial_m$ is compatible with relations \eqref{gW}.

Pulling the skew derivation $(\partial_{n}, \sigma_{\varphi^{-1}(\mu^{-1}))})$ in $R(\varphi(a), \varphi^{-1})$ back to $R(a,\varphi)$ through the $x$-$y$ symmetry one concludes  that $(\partial_{-n}, \sigma_\mu)$ \eqref{monomial.-n} is a skew derivation in $R(a,\varphi)$.
 
 Since $\bar\partial_b$ is given on $R$ by a $\sigma$-twisted commutator with $b$,  $(\bar\partial_b\!\!\mid_R ,\sigma)$ is a skew derivation. Note further that since $\sigma(a)=a$ and $a$ is a central element, $\bar\partial_b(a) = \bar\partial_b(\varphi(a)) =0$. Let us write $s = \left( \mu^{-1} b - \varphi(b) \right) $, so that $\bar\partial_b(x) = sx$, observe that $\bar\partial_b(y) = -\varphi^{-1}(\mu s) y$, and compute
$$
\bar\partial_b(xy) = s x \varphi^{-1}(\mu) y - x \varphi^{-1}(\mu s) y =0, \qquad \bar\partial_b(yx) = -\varphi^{-1}(\mu s) y \mu^{-1} x  + y sx =0,
$$
by \eqref{gW}, as required. Next, for all $r\in R$,
\begin{eqnarray*}
\bar\partial_b(xr - \varphi(r)x) &=&(s\varphi(\sigma(r))+\varphi\left(b\sigma(r)-rb\right)-\mu^{-1}\left(b\sigma(\varphi(r))-\varphi(r)b\right)-\varphi(r)s)x \\
&=& \left(s \varphi\left(\sigma(r)\right)+\left(\varphi(b)-\mu^{-1}b\right)\varphi\left(\sigma(r)\right)-\varphi(r)\left(\varphi(b)- \mu^{-1}b\right)-\varphi(r)s\right)x\\
&=& 0.
\end{eqnarray*}
 
 Finally, 
 \begin{eqnarray*}
\bar\partial_b(yr - \varphi^{-1}(r)y) &=& (-\varphi^{-1}(\mu s)\varphi^{-1}\left(\sigma(r)\right)+\varphi^{-1}\left(b\sigma(r) - rb\right)- b\sigma\left(\varphi^{-1}(r)\right)\varphi^{-1}(\mu)\\
&& +\varphi^{-1}(r)b \varphi^{-1}(\mu)+\varphi^{-1}(r)\varphi^{-1}(\mu s))y\\
&=& (-\varphi^{-1}(\mu s)\varphi^{-1}\left(\sigma(r)\right)+(\varphi^{-1}(b)- b \varphi^{-1}(\mu))\varphi^{-1}\left(\sigma(r)\right)-\\
&& \varphi^{-1}(r)(\varphi^{-1}(b)-b \varphi^{-1}(\mu))+\varphi^{-1}(r)\varphi^{-1}(\mu s))y
= 0,
\end{eqnarray*}
by using \eqref{gW} and \eqref{sigma.phi}.
This completes the proof  of the theorem. 
\end{proof}

\begin{remark}\label{rem.regular}
Note that the existence of a regular element of $R_\sigma ^R$ implies that, for all $z$ in the centre of $R$, $\sigma(z) =z$.
\end{remark}

\begin{remark}\label{rem.fixed}
Since the automorphism $\sigma$ commutes with $\varphi$ and $\sigma(a) =a$, the generalized Weyl algebra $R(a,\varphi)$ can be restricted to $S(a,\varphi)$, where  
$$
S := \{s \in R\; |\; \sigma(s) =s\}\subseteq R,
$$
is the fixed point subalgebra of $R$. If also $\sigma(\mu) = \mu$ (which e.g.\ is necessarily the case if there is a regular element in $R_\sigma^R$, see Remark~\ref{rem.regular}), then $\sigma_\mu$, restricted to $S(a,\varphi)$,  is the degree-counting extension of the identity automorphism of $S$ of coarseness  $\mu$. In this case the skew derivations listed in Theorem~\ref{thm.Weyl.deriv} restrict to skew derivations on $S(a,\varphi)$.
\end{remark}

\begin{definition}\label{definition.elementary}
The skew derivations listed in equations \eqref{monomials} will be referred to as {\em elementary}. The integer index $m$ of $\partial_m$ is called a {\em weight}.
\end{definition}

The construction of Theorem~\ref{thm.Weyl.deriv} can be given a cohomological interpretation
\begin{corollary}\label{cor.Hoch}
In the set-up of Theorem~\ref{thm.Weyl.deriv}, for all $m$, the assignment $\alpha_m\mapsto \partial_m$ induces an injective map
$$
HH^1_{\varphi^m\circ\sigma;\, \varphi^{m-1}(\mu),\varphi}(R) \longrightarrow HH^1_{\sigma_\mu}\left(R(a,\varphi)\right),
$$
of (doubly in the domain) twisted Hochschild cohomology groups.
\end{corollary}
\begin{proof}
Since $(\alpha_m, \varphi^m\circ\sigma)$ is a skew derivation, it is an element of $HC^1(R,R_{\varphi^m\circ \sigma})$, the first of conditions \eqref{q-deriv} implies that $\alpha_m \in HC^1_{\varphi^m\circ\sigma;\, \varphi^{m-1}(\mu),\varphi}(R)$. If $\alpha_m$ is inner with respect to $s \in HC^0_{\varphi^m\circ\sigma;\, \varphi^{m-1}(\mu),\varphi}(R)$, i.e.\ an element of $R$ such that $s = \varphi^{m}(\mu)\varphi(s)$, then one easily checks that $\partial_m$ is inner with respect to $sx^{m}$. This proves the existence of the map between cohomology groups. 

If $\partial_m$ is inner, then for all $r\in R$,
\begin{eqnarray*}
\partial_m(r)  &=& \sum_k s_kx^k \sigma(r) + \sum_l r_ly^l\sigma(r) - \sum_krs_kx^k - \sum_lrr_ly^l\\
&=& \sum_k\left(s_k\varphi^k(\sigma(r)) - rs_k\right)x^k +  \sum_l\left(r_l\varphi^{-l}(\sigma(r)) - rr_l\right)y^l 
= \alpha_m(r) x^m,
\end{eqnarray*}
which implies that $r_l =0$ for all $l$, and $s_k =0$ for all $k\neq m$. Hence
$$
\alpha_m(r) = s_m \varphi^m(\sigma(r)) - rs_m,
$$
i.e.\ $\alpha_m$ is inner. Therefore, the constructed map is an additive monomorphism, as stated.
\end{proof}

The proof of Theorem~\ref{thm.Weyl.deriv} provides one with almost full classification of skew derivations of a generalized Weyl algebra.
\begin{corollary}\label{cor.Weyl.deriv}
Let $R(a,\varphi)$ be a generalized Weyl algebra  with $a\in R$ neither zero nor a zero divisor, and let $\sigma$ be an automorphism of $R$ commuting with $\varphi$ and fixing $a$. Let $\sigma_\mu$ be the degree-counting extension of $\sigma$ of coarseness $\mu$. If  $(\partial, \sigma_\mu)$ is a skew derivation of $R(a,\varphi)$ such that either
\begin{rlist}
\item $\partial(R) \subset R(a,\varphi)_+$  and  $\partial(x) =0$, or 
\item  $\partial(R) \subset R(a,\varphi)_-$  and $\partial(y) =0$, 
\end{rlist}
then it is of the type described in Theorem~\ref{thm.Weyl.deriv}.
\end{corollary}
\begin{proof}
In the first case necessarily,
$$
\partial(r) = \sum_{m=1}^M\alpha_m(r)\, x^m,  \qquad \mbox{for all $r\in R$}. 
$$
Setting $\partial(x) =0$, assuming the general form
$$
\partial(y) = \sum_{i=0} c_i x^i +\sum_{j=1} d_j y^j,
$$
and demanding $\partial(yx -a) = 0$, one obtains:
\begin{eqnarray*}
\sum_{m=1}^M\alpha_m(a)\, x^m &=&  \sum_{i=0} \varphi^i\left(\mu^{-1}\right)c_i x^{i +1}+\sum_{j=1} \varphi^{-j}\left(\mu^{-1}\right) d_j y^jx\\
&=& \sum_{i=0} \varphi^i\left(\mu^{-1}\right) c_i x^{i +1}+\sum_{j=1} \varphi^{-j}\left(\mu^{-1}\right) d_j \varphi^{-j+1}(a)y^{j-1}.
\end{eqnarray*}
This implies that $d_j =0$, for all $j$, while $c_{m-1} = \varphi^{m-1}\left(\mu\right) \alpha_m(a)$.  The proof of Theorem~\ref{thm.Weyl.deriv} affirms the necessity of conditions listed in Theorem~\ref{thm.Weyl.deriv}. The other case is deduced by the $x$-$y$ symmetry.
\end{proof}

So far we have made no restrictions on the central unit $\mu\in R$, which determined the coarseness of the degree-counting automorphism. In all examples we have in mind, however, where typically $R$ is an algebra over a field  and $\mu$ is a scalar parameter, $\mu$ is a central element in the whole of the generalized Weyl algebra $R(a,\varphi)$, or equivalently, $\varphi(\mu) = \mu$. Furthermore, if $\mu$ is scalar, also $\sigma(\mu) = \mu$. Having these typical applications in mind and  to avoid undue complications in the formulae we make this assumption in the following proposition.

\begin{proposition}\label{prop.q-deriv}
Let $R(a,\varphi)$ be a generalized Weyl algebra and let $\sigma$ be an automorphism of $R$ commuting with $\varphi$ and fixing $a$.  Let $\sigma_\mu$ be the degree-counting extension of $\sigma$ of coarseness $\mu$  such that $\varphi(\mu) = \mu= \sigma(\mu)$, and let $Q$ be a central unit in $R$ such that $\varphi(Q) = \sigma(Q) = Q$. If $b\in R_\sigma^R$, then the skew derivation $(\partial, \sigma_\mu)$ \eqref{delta} is a skew $Q$-derivation if and only if
\begin{blist}
\item For all $i= -N, \ldots, M$, $(\alpha_i, \varphi^i\circ\sigma)$ are skew $Q$-derivations;
\item $\sigma\left(c - \varphi(b) +\mu^{-1} b\right) = Q\left(c - \varphi(b) +\mu^{-1} b\right)$.
\end{blist}
If $b\not\in R_\sigma^R$, then the skew derivation $(\partial, \sigma_\mu)$ \eqref{delta} is a skew $Q$-derivation if and only if, in addition to (a) and (b), $Q=1$.
\end{proposition}
\begin{proof}
First note that $(\alpha_i, \varphi^i\circ\sigma)$ is a skew $Q$-derivation if and only if 
\begin{equation}\label{alpha.skew}
\sigma\circ\alpha_i\circ\sigma^{-1} = \mu^i Q\alpha_i.
\end{equation}
Indeed, in view of the repeated use of \eqref{q-deriv},
$$
\varphi^i\circ \sigma\circ \alpha_i\circ\sigma^{-1}\circ\varphi^{-i} = \mu^{-i} \, \sigma\circ \alpha_i\circ\sigma^{-1},
$$
hence $\varphi^i\circ \sigma\circ \alpha_i\circ\sigma^{-1}\circ\varphi^{-i} = Q \alpha_i$ if and only if the condition \eqref{alpha.skew} is fulfilled. Observe that, since $\sigma$ fixes $a$ and commutes with $\varphi$, conditions \eqref{alpha.skew} imply,
\begin{equation}\label{alpha.skew.a}
\sigma\left(\alpha_i \left(\varphi^k\left(a\right)\right)\right) = \mu^{i} Q\, \alpha_i \left(\varphi^k\left(a\right)\right),
\end{equation}
for all $i,k\in \ZZ$.

To prove that skew derivation $(\partial, \sigma_\mu)$ \eqref{delta} is a skew $Q$-derivation, we need to check that $ \sigma_\mu\circ\partial\circ\sigma_\mu^{-1} =  Q\partial$ is satisfied. First, suppose that (a) and (b) hold, for all $b\in R_\sigma^R$, we consider a general skew derivation of the form \eqref{delta}
$$
\partial(r) = \sum_{m=0}^M\alpha_m(r)\, x^m +\sum_{n=1}^N\alpha_{-n}(r)\,  y^n  ,
$$
$$
\partial(x) = c' x + \sum_{n=1}^N \varphi \left(\alpha_{-n}\left(a\right)\right) y^{n-1},
$$
$$
\partial(y) = \mu \left(\frac{\alpha_0(a)}a - \varphi^{-1}(c') \right) y + \mu\, \sum_{m=1}^M \alpha_m(a)\, x^{m-1},
$$
where $c' = c - \varphi(b) + \mu^{-1} b \in R_\sigma^R$. Note that since in this case $b\in R_\sigma^R$, the contribution coming from $\bar\partial_b$ can be (and has been) absorbed into other elementary derivations.
Then, using \eqref{alpha.skew} we compute
\begin{eqnarray*}
\sigma_\mu\circ\partial\circ\sigma_\mu^{-1}(r)&=& \sum_{m=0}^M \sigma\left(\alpha_m(\sigma^{-1}(r))\right)\, \sigma_\mu(x^m) +\sum_{n=1}^N \sigma\left(\alpha_{-n}(\sigma^{-1}(r))\right)\,  \sigma_\mu(y^n) \\
&=& \sum_{m=0}^M \mu^m Q \alpha_m(r) \mu^{-m} x^m+ \sum_{n=1}^N \mu^{-n} Q \alpha_{-n}(r) \mu^n y^n \\
&=& Q \left(\sum_{m=0}^M\alpha_m(r)\, x^m +\sum_{n=1}^N\alpha_{-n}(r)\,  y^n \right)
= Q\partial(r).
\end{eqnarray*}
Next, in view of \eqref{alpha.skew.a} and (b),
\begin{eqnarray*}
\sigma_\mu\circ\partial\circ\sigma_\mu^{-1}(x)&=& \mu\, \sigma_\mu\left(\partial(x)\right)\\
&=& \mu\, (\sigma\left(c'\right)  \sigma_\mu(x) + \sum_{n=1}^N \sigma\left(\varphi \left(\alpha_{-n}\left(a\right)\right)\right) \sigma_\mu(y^{n-1}))\\
&=& \mu(Q c' \mu^{-1} x+ \sum_{n=1}^N \mu^{-n} \sigma\left(\alpha_{-n}(\varphi(a))\right) \mu^{n-1} y^{n-1})\\
&=& Q c'  x+ \sum_{n=1}^N \mu^{-1} Q \alpha_{-n}(\varphi(a)) y^{n-1}\\
&=&  Q c'  x+ \sum_{n=1}^N Q \varphi(\alpha_{-n}(a)) y^{n-1}
= Q \partial(x).
\end{eqnarray*}
Furthermore, 
\begin{eqnarray*}
\sigma_\mu\circ\partial\circ\sigma_\mu^{-1}(y)&=& \mu^{-1} \sigma_\mu\left(\partial(y)\right)\\
&=& \mu^{-1}(\mu \sigma\left(\frac{\alpha_0(a)}a - \varphi^{-1}(c') \right) \sigma_\mu(y) + \mu\, \sum_{m=1}^M \sigma\left(\alpha_m(a)\right)\, \sigma_\mu(x^{m-1})\\
&=& \mu Q \left(\frac{\alpha_0(a)}a - \varphi^{-1}(c')\right) y + \sum_{m=1}^M  \mu^{m}Q \alpha_m(a)\,  \mu^{-m+1}x^{m-1}
= Q \partial(y),
\end{eqnarray*}
hence we conclude that $(\partial, \sigma_\mu)$ \eqref{delta} is a skew $Q$-derivation. On the other hand, if $(\partial, \sigma_\mu)$ is a skew $Q$-derivation, then all of the elementary skew derivations \eqref{monomials} are skew $Q$-derivations. Noting this for \eqref{monomial.0}, \eqref{monomial.m}, \eqref{monomial.-n} we obtain (a), while the skew $Q$-derivation property of \eqref{monomial.0} and \eqref{monomial.00} imply (b).

In the second part, since $b\not\in R_\sigma^R$ the inner derivation $\partial_{b}(r)$ in $\partial(r)$  is not zero, and hence it is a skew $1$-derivation. Therefore, by a similar calculation as in the first part, $(\partial, \sigma_\mu)$ \eqref{delta} is a skew $Q$-derivation if and only if $Q=1$ in addition to (a) and (b).
\end{proof}

\begin{proposition}\label{prop.grading}
Let $R$ be a $\ZZ$-graded ring and consider $R(a,\varphi)$ as a $\ZZ$-graded ring with the $(d,k)$-type grading. Let $\sigma$ be an automorphism of the graded ring $R$ commuting with $\varphi$ and fixing $a$. Let $\sigma_\mu$ be the degree-counting extension of $\sigma$ of coarseness $\mu$ of degree 0. Let $(\partial, \sigma_\mu)$ be the skew derivation associated as in Theorem~\ref{thm.Weyl.deriv} to the data $\alpha_i$, $b$, $c$. Then $\partial$ is a map of degree $l$  if and only if, 
\begin{equation}\label{grading.conditions}
\deg(\alpha_i) = l -ik +\frac{i - |i|}{2} d, \qquad \deg(b) = \deg(c) =l.
\end{equation}
\end{proposition}
\begin{proof}
Indeed, we need here to check $\deg(\alpha_i) $ in three cases where $i$ is zero, positive and negative respectively. Notice that the last term of the first equality in \eqref{grading.conditions} will disappear in the first two cases.  Suppose that $\partial$ is a map of degree $l$, then \eqref{monomials} gives:
$$
\deg(\alpha_0)=\deg(\partial_0)=l, \quad \deg(c)=l,
$$
$$
\deg(\alpha_i)=\deg(\alpha_m)=\deg(\partial_m)-\deg(x^{m})=l-mk,
$$
and
$$
\deg(\alpha_i)=\deg(\alpha_{-n})=\deg(\partial_{-n})-\deg(y^{n})=l-n(d-k)=l+nk-nd.
$$
Put together this gives us most of \eqref{grading.conditions}. Finally, in the view of  \eqref{monomial.00} and $\deg(\sigma)=0$,  we can observe that $\deg(b)=l$. On the other hand, if \eqref{grading.conditions} holds, then \eqref{monomials} clearly implies that $\partial$ is a map of degree $l$.  
\end{proof}

Additional classes of skew derivations can be constructed for generalized Weyl algebras associated to automorphisms of finite orders.

\begin{proposition}\label{prop.finite.order}
Let $R(a,\varphi)$ be a generalized Weyl algebra and let $\sigma$ be an automorphism of $R$ commuting with $\varphi$ and fixing $a$. Let $\sigma_\mu$ be the degree-counting extension of $\sigma$ of coarseness $\mu$.  Assume that $\varphi$ has a finite order $D$, i.e.\ 
\begin{equation}\label{order}
\varphi^D = \id,
\end{equation} 
and let
$$
(\alpha_i ,\sigma)_{i=-N}^M, 
$$
be skew derivations on $R$ such that, for all $i = -N,\ldots, M$,
\begin{equation}\label{q-deriv.finite}
\alpha_i\circ \varphi= \mu\, \varphi \circ \alpha_i  , 
\quad  \quad \alpha_i(a) \in R_\sigma^R.
\end{equation}
For all $b_m, c_n\in R_\sigma^R$, set
\begin{subequations}\label{delta.f-o}
\begin{equation}\label{delta.r.f-o}
\partial(r) = \sum_{m=1}^M\alpha_m(r)\, x^{mD} +\sum_{n=1}^N\beta_{n}(r)\,  y^{nD}, \qquad \mbox{for all $r\in R$},
\end{equation}
\begin{equation}\label{delta.x.f-o}
\partial(x) = \sum_{m=1}^Mb_mx^{mD+1} + \mu^{-1}\sum_{n=1}^N \left(\alpha_{-n}(\varphi(a)) - \varphi\left(c_n\right)a\right) y^{nD-1},
\end{equation}
\begin{equation}\label{delta.y.f-o}
\partial(y) = \varphi^{-1}(\mu) \sum_{m=1}^M\left(\alpha_m(a)  - \varphi^{-1}\left( b_m\right) a \right) x^{mD-1} +\sum_{n=1}^N c_n\, y^{nD+1}.
\end{equation}
\end{subequations}
Then $\partial$ extends to a skew derivation $(\partial, \sigma_\mu)$ on $R(a,\varphi)$.
\end{proposition}
\begin{proof}
As was the case for Theorem~\ref{thm.Weyl.deriv}, we prove that, for all $m$ and $n$, the following maps extend to the derivations of $R(a,\varphi)$,
$$
\partial_m(r) = \alpha_m(r)\, x^{mD}, \quad \!\!\! \partial_m(x) = b_mx^{mD+1}, \quad \!\!\! \partial_m(y) = \varphi^{-1}(\mu) \left(\alpha_m(a)  - \varphi^{-1}\left( b_m\right) a \right) x^{mD-1},
$$
$$
\partial_{-n}(r) = \alpha_{-n}(r)\,  y^{nD}, \quad\!\! \partial_{-n}(x) =\mu^{-1}\left(\alpha_{-n}(\varphi(a)) - \varphi\left(c_n\right)a\right) y^{nD-1}, \quad \!\! \partial_{-n}(y) =c_n\, y^{nD+1}.
$$

First, since $(\alpha_m,\sigma)$ is a skew derivation on $R$, and $\sigma_\mu$ restricted to $R$ is equal to $\sigma$, $\partial_m$ satisfies the $\sigma_\mu$-twisted Leibniz rule. We need to check that $\partial_m $ extended to the whole of $R(a,\varphi)$ by the $\sigma_\mu$-twisted Leibniz rule preserves all relations \eqref{gW}. In view of \eqref{order} we  make constant use of the fact that all powers of $\varphi$ can be calculated modulo $D$ and start by computing
\begin{eqnarray*}
\partial_m(xy-\varphi(a))\!\!\!&=& \!\!\!b_m x^{mD+1}  y \mu+ x \varphi^{-1}(\mu)(\alpha_m(a)-\varphi^{-1}(b_m)a) x^{mD-1}-\alpha_m(\varphi(a)) x^{mD}\\
&=& \!\!\! \mu b_m \varphi(a) x^{mD} + \mu\varphi(\alpha_m(a)) x^{mD} - \mu b_m \varphi(a) x^{mD}
 - \mu \varphi(\alpha_m(a)) x^{mD}
 =0,
\end{eqnarray*}
where \eqref{gW}, \eqref{q-deriv.finite} and the centrality of $\mu\in R$ were used.
In a similar way one easily finds that,
\begin{eqnarray*}
\partial_m(yx-a)\!\!\!&=& \!\!\!\varphi^{-1}(\mu)\left(\alpha_m\left(a\right)-\varphi^{-1}(b_m)a\right)x^{mD-1} \mu^{-1} x+ y b_m x^{mD+1} - \alpha_m(a)x^{mD}\\
&=&\!\!\!\alpha_m(a)x^{mD}- \varphi^{-1}(b_m)a x^{mD}+\varphi^{-1}(b_m) a x^{mD}-\alpha_m(a) x^{mD}
=0.
\end{eqnarray*}
Furthermore, for all $r\in R$,
\begin{eqnarray*}
\partial_m(xr-\varphi(r)x)\!\!\!&=& \!\!\!b_m x^{mD+1}\sigma(r)+x \alpha_m(r)x^{mD}-\alpha_m(\varphi(r))x^{mD}\mu^{-1}x-\varphi(r)b_m x^{mD+1}\\
&=&\!\!\! b_m \varphi(\sigma(r))x^{mD+1}+\varphi(\alpha_m(r))x^{mD+1}-\mu^{-1}\alpha_m(\varphi(r))x^{mD+1}\\
&&\!\!\!-\varphi(r)b_mx^{mD+1}=b_m\varphi(\sigma(r))x^{mD+1}-\varphi(r)b_mx^{mD+1}=0,
\end{eqnarray*}
where the first equality follows by the definition of $\partial_m$ via the twisted Leibniz rule. The second one follows by \eqref{gW} and the centrality of $\mu$, the third one by \eqref{q-deriv.finite}, while the last one follows by the fact that $b_m \in R_\sigma^R$.
In a similar way,
\begin{eqnarray*}
\partial_m(yr-\varphi^{-1}(r)y)\!\!\!&=& \!\!\! \varphi^{-1}( \mu) \left(\alpha_m(a)  - \varphi^{-1}\left( b_m\right) a \right) x^{mD-1}\sigma(r)+y\alpha_m(r)x^{mD}\\
&& \!\!\! \!\!\! - \alpha_m(\varphi^{-1}(r))x^{mD} y\mu -\varphi^{-1}(r) \varphi^{-1}( \mu)\left(\alpha_m(a)  - \varphi^{-1}\left( b_m\right) a \right) x^{mD-1}\\
\!\!\!&=& \!\!\!\varphi^{-1}( \mu)\alpha_m(a)\varphi^{-1}(\sigma(r))x^{mD-1}-\varphi^{-1}( \mu) \varphi^{-1}(b_m)a \varphi^{-1}(\sigma(r)) x^{mD-1}\\
 && \!\!\! \!\!\!+\varphi^{-1}(\alpha_m(r))a x^{mD-1} -\varphi^{-1}( \mu) \alpha_m(\varphi^{-1}(r))\varphi^{mD}(a) x^{mD-1}\\
&&\!\!\! \!\!\!-\varphi^{-1}( \mu) \varphi^{-1}(r)\alpha_m(a)x^{mD-1}+\varphi^{-1}( \mu) \varphi^{-1}(r) \varphi^{-1}(b_m)a x^{mD-1}=0.
\end{eqnarray*}
Thus, $\partial_m$ vanishes on all generators of the ideal in $R\langle x,y\rangle$ that defines $R(a,\varphi)$, hence $\partial_m$ extends to a $\sigma_\mu$-twisted derivation to the whole of $R(a,\varphi)$.
The fact that  $\partial_{-n}$ extends to the whole of $R(a,\varphi)$ as a $\sigma_\mu$-derivations follows by the $x$-$y$ symmetry.
\end{proof}

\section{Orthogonal pairs of skew derivations on generalized Weyl algebras}\label{sec.ortho}\setcounter{equation}{0}
The orthogonality of a system of skew derivations on  a given ring $A$ relies heavily on the structure of $A$, and -- in general -- very little can be said even in the case of rather explicitly defined generalized Weyl algebras over $R$ if the  ring $R$ is not specified. The cases of $R$ being a polynomial ring in one and two variables are discussed in some detail in \cite{Brz:dif}. Here, rather than specifying $R$, we would like to concentrate on a general case, and in such a case at least some examples of pairs of orthogonal skew derivations (included in the families described in Theorem~\ref{thm.Weyl.deriv}) can be given. We start with the following simple observation.
\begin{lemma}\label{lemma.ortho}
Let $(\partial_i, \sigma_i)_{i=1}^n$ be a family of skew derivations on a ring $A$. If there exist 
$\{b_1,b_2,\ldots, b_n\} \subset A$
such that
\begin{equation}\label{ortho.lemma}
A\partial_i(b_{i})A = A, \qquad \partial_k(b_{i}) = 0, \quad \mbox{for all $i\neq k$},
\end{equation}
then $(\partial_i, \sigma_i)_{i=1}^n$  is an orthogonal system of skew derivations.
\end{lemma}
\begin{proof}
The fact that the ideal generated by $\partial_i(b_i)$ is equal to $A$ is equivalent to the existence of 
finite subsets $\{a_{i\,t}\}, \{c_{i\,t}\}$ of elements of $A$  such that, 
$$
1 = \sum_t a_{i\,t}\partial_i(b_{i}) c_{i\, t}  = \sum_t  a_{i\,t}\partial_i\left(b_{i} \sigma^{-1}_i\left(c_{i\, t}\right)\right) -   \sum_t  a_{i\,t}b_{i} \partial_i\left(\sigma^{-1}_i\left(c_{i\, t}\right)\right),
$$
where the second equality follows by the $\sigma_i$-twisted Leibniz rule. This gives condition \eqref{ortho} with $i=k$. If $i\neq k$, then,
\begin{eqnarray*}
 \sum_t  a_{i\,t}\partial_k\left(b_{i} \sigma^{-1}_i\left(c_{i\, t}\right)\right) &-&   \sum_t  a_{i\,t}b_{i} \partial_k\left(\sigma^{-1}_i\left(c_{i\, t}\right)\right)\\
 & = & \sum_t  a_{i\,t}b_{i}  \partial_k\left(\sigma^{-1}_i\left(c_{i\, t}\right)\right) -   \sum_t  a_{i\,t}b_{i} \partial_k\left(\sigma^{-1}_i\left(c_{i\, t}\right)\right) =0,
\end{eqnarray*}
by the $\sigma_k$-twisted Leibniz rule and since $\partial_k(b_i) =0$. This confirms that \eqref{ortho} holds also for $i\neq k$.
 \end{proof}
 
In the following, by saying  that two elements $r,s\in R$ are {\em coprime} we will mean that the ideals generated by them are coprime, i.e.\ that
$$
RrR + RsR =R.
$$
\begin{proposition}\label{prop.ortho}
Let $R(a,\varphi)$ be a generalized Weyl algebra and let $\sigma$, $\bar\sigma$ be automorphisms of $R$ commuting with $\varphi$ and fixing $a$. Let $\sigma_\mu$,  $\bar\sigma_{\bar\mu}$ be  their degree-counting extensions with respective coarsenss $\mu$ and  $\bar\mu$. Choose a positive integer $N$ such that $a$ is coprime with $\varphi^i(a)$, for all $i\in \{1,2,\ldots, 2N-1\}$, fix $m,n \in \{0,1,\ldots, N\}$ and consider the following data:
\begin{blist}
\item A skew derivation $(\alpha, \sigma\circ \varphi^{m+1})$ of $R$ such that
\begin{rlist}
\item $\alpha(a)$ is in the centre of $R$,
\item $\alpha(a)$ is coprime with  $\varphi^j(a)$, $j\in \{-m-1,-m, \ldots, 0, m+1,m+2,\ldots, 2m\}$ and with $\varphi^{-m}(\alpha(a))$,
\item $\alpha\circ \varphi = \varphi^{m+1}(\mu)\, \varphi\circ\alpha$.
\end{rlist}
\item A skew derivation $(\bar\alpha, \bar\sigma\circ \varphi^{-n-1})$ of $R$ such that
\begin{rlist}
\item $\bar\alpha(a)$ is in the centre of $R$,
\item $\varphi^{n+1}(\bar\alpha(a))$ is  coprime with $\varphi^j(a)$, $j\in \{-n-1,-n, \ldots, 0, n+1,n+2,\ldots, 2n\}$ and with $\varphi(\bar\alpha(a))$,
\item $\bar\alpha\circ \varphi = \varphi^{-n-1}(\bar\mu)\, \varphi\circ\bar\alpha$.
\end{rlist}
\end{blist}
Then the elementary skew derivations $(\partial,\sigma_\mu)$  and $(\bar\partial,\bar\sigma_{\bar\mu})$ of $R(a,\varphi)$ associated to $\alpha$, $\bar\alpha$ as in Theorem~\ref{thm.Weyl.deriv} form an orthogonal pair.
\end{proposition}
\begin{proof}
Explicitly, the elementary weight $m+1$ and $-n-1$ respectively skew derivations $\partial$, $\bar\partial$ are given by
\begin{subequations}\label{partials}
\begin{equation}\label{partial}
\partial(r) = \alpha (r) x^{m+1}, \qquad \partial(x) = 0, \qquad \partial(y) =  \alpha (a)\, x^{m}\mu\, ,
\end{equation}
\begin{equation}\label{bar.partial}
\bar\partial(r) = \bar\alpha (r) y^{n+1}, \qquad \bar\partial(x) = \varphi\left(\bar\alpha\left(a\right)\right) y^{n}, \qquad \bar\partial(y) = 0,
\end{equation}
\end{subequations}
for all $r\in R$, and then extended to the whole of $R(a,\varphi)$ by the twisted Leibniz rules. We will show that $\partial(y)$ and $\bar\partial(x)$ generate ideals both equal to $R(a,\varphi)$ and then use Lemma~\ref{lemma.ortho} to conclude that $(\partial,\sigma_\mu)$  and $(\bar\partial,\bar\sigma_{\bar\mu})$ form an orthogonal pair. Observe that, in view of \eqref{gW} and the centrality of $\alpha(a)$,
$$
y^m\partial(y) = \mu \varphi^{-m}\left(\alpha\left(a\right)\right) \varphi^{-m+1}\left(a\right)\cdots \varphi^{-1}(a) a,
$$
and
$$
\partial(y)y^m = \varphi^{m}(\mu) \alpha(a)\varphi(a)\varphi^2(a)\cdots \varphi^m(a).
$$
Hence the ideal generated by $\partial(y)$ is equal to the whole of $R(a,\varphi)$, provided
\begin{equation}\label{ideal}
R\varphi^{-m}\left(\alpha\left(a\right)\right) \varphi^{-m+1}\left(a\right)\cdots \varphi^{-1}(a) a + R\alpha(a)\varphi(a)\varphi^2(a)\cdots \varphi^m(a) = R.
\end{equation}
In a similar way,
$$
x^n\bar\partial(x) = \varphi^{n+1}\left(\bar\alpha\left(a\right)\right) \varphi^{n}\left(a\right)\cdots \varphi^{2}(a) \varphi(a),
$$
and
$$
\bar\partial(x)x^n =  \varphi\left(\bar\alpha(a)\right)a\varphi^{-1}(a)\cdots \varphi^{-n+1}(a).
$$
Hence the ideal generated by $\bar\partial(x)$ is equal to the whole of $R(a,\varphi)$, provided
\begin{equation}\label{ideal.bar}
R\varphi\left(\bar\alpha(a)\right)a\varphi^{-1}(a)\cdots \varphi^{-n+1}(a) + R\varphi^{n+1}\left(\bar\alpha\left(a\right)\right) \varphi^{n}\left(a\right)\cdots \varphi^{2}(a) \varphi(a) = R.
\end{equation}

Since $a$ is coprime with all $\varphi^i(a)$, $i\in \{1,2,\ldots, 2N-1\}$, $\varphi^{-i}(a)$ is coprime with $\varphi^{j}(a)$, for all $i \in \{0, 1, \ldots, N-1\}$ and $j\in \{1,2,\ldots , N\}$, and hence, 
$$
R =  Ra + R\varphi^j(a) = \left(R\varphi^{-1}(a) + R\varphi^j(a)\right) a + R\varphi^j(a)  = R\varphi^{-1}(a)a + R\varphi^j(a),
$$
where the last equality is a consequence of $R\varphi^j(a) a \subseteq R\varphi^j(a)$. Next,
\begin{eqnarray*}
 R &=&  R\varphi^{-1}(a)a + R\varphi^j(a) = \left(R\varphi^{-2}(a) + R\varphi^j(a)\right)\varphi^{-1}(a) a + R\varphi^j(a) \\
 & = & R\varphi^{-2}(a)\varphi^{-1}(a)a + R\varphi^j(a).
\end{eqnarray*}
Repeating this sufficiently many times, one concludes that
\begin{equation}\label{-i.j}
R = R\varphi^{-i}(a)\cdots \varphi^{-2}(a)\varphi^{-1}(a)a + R\varphi^j(a),
\end{equation}
for all $i \in \{0, 1, \ldots, N-1\}$ and $j\in \{1,2,\ldots , N\}$. Similarly, starting with
$R = Ra + R\alpha(a)$, and using that $\alpha(a)$ is coprime with $\varphi^{-m}(\alpha(a))$ and all $\varphi^{-i}(a)$, where $i \in \{0, 1, \ldots, m-1\}$, by the same arguments one obtains that
\begin{equation}\label{-m.alpha}
R = R\varphi^{-m}(\alpha(a))\varphi^{-m+1}(a)\cdots \varphi^{-2}(a)\varphi^{-1}(a)a + R\alpha(a).
\end{equation}
Since $\alpha(a)$ is coprime with $\varphi^j(a)$, for all $j\in \{m+1, \ldots, 2m\}$, $\varphi^{-m}(\alpha(a))$ is coprime with $\varphi^j(a)$, for all $j\in \{1, \ldots, m\}$. Bearing in mind that $m\leq N$, \eqref{-i.j} implies that
\begin{eqnarray}\label{alpha.phij}
R &=& \left(R\varphi^{-m}(\alpha(a)) +R\varphi^j(a)\right)  \varphi^{-i}(a)\cdots \varphi^{-2}(a)\varphi^{-1}(a)a + R \varphi^j(a)\nonumber\\
&=& R\varphi^{-m}(\alpha(a))\varphi^{-m+1}(a)\cdots \varphi^{-2}(a)\varphi^{-1}(a)a + R \varphi^j(a),
\end{eqnarray}
for all $j\in \{1, \ldots, m\}$.  Starting with \eqref{-m.alpha} and then repeatedly using \eqref{alpha.phij} we thus obtain
\begin{eqnarray*}
R &=& R\varphi^{-m}(\alpha(a))\varphi^{-m+1}(a)\cdots \varphi^{-2}(a)\varphi^{-1}(a)a + R\alpha(a)\\
&=& R\varphi^{-m}(\alpha(a))\varphi^{-m+1}(a)\cdots \varphi^{-2}(a)\varphi^{-1}(a)a \\
&&+ \left(R\varphi^{-m}(\alpha(a))\varphi^{-m+1}(a)\cdots \varphi^{-2}(a)\varphi^{-1}(a)a + R \varphi(a)\right)\alpha(a)\\
&=& R\varphi^{-m}(\alpha(a))\varphi^{-m+1}(a)\cdots \varphi^{-2}(a)\varphi^{-1}(a)a +  R \varphi(a)\alpha(a)\\
&=& R\varphi^{-m}(\alpha(a))\varphi^{-m+1}(a)\cdots \varphi^{-2}(a)\varphi^{-1}(a)a \\
&& +  \left(R\varphi^{-m}(\alpha(a))\varphi^{-m+1}(a)\cdots \varphi^{-2}(a)\varphi^{-1}(a)a + R \varphi^2(a)\right) \varphi(a)\alpha(a)\\
&=& R\varphi^{-m}(\alpha(a))\varphi^{-m+1}(a)\cdots \varphi^{-2}(a)\varphi^{-1}(a)a +  R \varphi^2(a)\varphi(a)\alpha(a)\\
&=& \ldots\\
&=& R\varphi^{-m}(\alpha(a))\varphi^{-m+1}(a)\cdots \varphi^{-2}(a)\varphi^{-1}(a)a +  R\varphi^m(a)\cdots  \varphi^2(a)\varphi(a)\alpha(a),
\end{eqnarray*}
i.e.\ the required equation \eqref{ideal}. Replacing $m$ by $n$ and $\alpha(a)$ by $\varphi^{n+1}(\bar\alpha(a))$ in the above arguments, one finds that also \eqref{ideal.bar} holds. Now Lemma~\ref{lemma.ortho} implies that $(\partial,\sigma_\mu)$ and  $(\bar\partial,\sigma_{\bar\mu})$ form an orthogonal pair.
\end{proof}

\begin{remark}\label{rem.ortho}
Note that since, for all central elements $r,s$ of $R$, $Rrs \subseteq Rr,\, Rs$, equality \eqref{ideal} implies hypothesis (a)(ii), while \eqref{ideal.bar} implies hypothesis (b)(ii) in Proposition~\ref{prop.ortho}.
\end{remark}

\section{skew derivations of the quantum disc and quantum plane algebras}\label{sec.disc}\setcounter{equation}{0}
In this section we apply the results of Section~\ref{sec.skew} to generalized Weyl algebras over a polynomial ring in one variable  associated to  linear polynomials and known as the {\em quantum disc algebra} and the {\em quantum plane} or the {\em quantum polynomial ring in two variables}. 

Let $\KK$ be a field and $q$ a non-zero element  of $\KK$. The {\em coordinate algebra of the quantum disc} $D_q(x,y)$ or the {\em quantum disc algebra} is a $\KK$-algebra generated by $x,y$ and the relation
\begin{equation}\label{q.disc}
xy - q\,yx =1-q.
\end{equation}
The {\em quantum polynomial ring in two variables} or the {\em quantum plane algebra} is a $\KK$-algebra $\KK_q[x,y]$ generated by $x,y$ and the relation
\begin{equation}\label{q.plane}
xy =  q\,yx.
\end{equation}
Both algebras have $\KK$-linear bases  given by all monomials $y^mx^n$. 
$D_q(x,y)$ and $\KK_q[x,y]$ are two of the simplest examples of a generalized Weyl algebra. Consider the following automorphism of the polynomial algebra in one variable,
\begin{equation}\label{auto.poly}
\varphi: \KK[h]\to \KK[h], \qquad f(h) \mapsto f(qh).
\end{equation}
Then 
\begin{equation}\label{disc.Weyl}
D_q(x,y) = \KK[h](1-h,\varphi) \quad \mbox{and} \quad \KK_q[x,y] = \KK[h](h,\varphi).
\end{equation}
Non-zero scalar multiples of the identity are the only units of the polynomial algebra $\KK[h]$. We choose such a multiple $\mu$. Since an automorphism $\sigma$ considered in Lemma~\ref{lemma.auto} should satisfy $\sigma(1-h) =1-h$, in the case of $D_q(x,y)$ or $\sigma(h) =h$ in the $\KK_q[x,y]$-case, $\sigma$ must be the identity automorphism. Thus $\sigma_\mu$ is fully determined by \eqref{sigma.mu}.
\begin{remark}\label{rem.linear}
Note that, up to isomorphism, $D_q(x,y)$ and $\KK_q[x,y]$ are the only two generalized Weyl algebras over $\KK[h]$ corresponding to the automorphism \eqref{auto.poly} and a linear polynomial. Indeed, the relations
$$
yx = \alpha +\beta\, h, \qquad xy = \alpha +q\beta\, h, \qquad \beta\neq 0,
$$
yield
$$
xy - qyx = (1-q)\alpha.
$$
If $\alpha =0$ we obtain $\KK_q[x,y]$, while if $\alpha\neq 0$, by rescaling the generators we obtain $D_q(x,y)$.
\end{remark}

The following lemma proves useful in calculations involving the basis  $y^mx^n$ of $D_q(x,y)$.
\begin{lemma}\label{lemma.disc}
For all $n\geq 1$,  the generators $x,y$ of the quantum disc algebra  $D_q(x,y)$ satisfy the following equalities: 
 \begin{equation}\label{xyn}
 x y^n -q^ny^nx = (1-q^n)y^{n-1}, \qquad x^n y -q^nyx^n = (1-q^n)x^{n-1}. 
 \end{equation}
 \end{lemma}
 \begin{proof}
This follows by the identification of $D_q(x,y)$ as a generalized Weyl algebra in \eqref{disc.Weyl} or can be proven from relations \eqref{q.disc} by induction.
 \end{proof}

\begin{proposition}\label{prop.disc.one}
Assume that a non-zero $q\in \KK$ is not a root of unity, and let $A$ be either
\begin{blist}
\item the disc algebra $D_q(x,y)$ or
\item the quantum polynomial ring $\KK_q[x,y]$.
\end{blist}
 Set $h = 1-yx$ if $A=D_q(x,y)$ or $h=yx$ if $A=\KK_q[x,y]$, and let $\mu$ be a  non-zero element of $ \KK$. 
\begin{zlist}
\item For all $f(h) \in \KK[h]$,  the map $\partial$ on generators of $A$ given by
\begin{equation}\label{disc.deriv.zero}
\partial(x) = f(h) x, \qquad \partial(y) = - \mu f(q^{-1}h)y,
\end{equation}
extends to a skew derivation $(\partial, \sigma_\mu)$ of $A$. These are the only $\sigma_\mu$-derivations such that $\partial(h) = 0$.
\item
If there exists $d\in \NN$ such that 
\begin{equation}\label{mu.q}
\mu = q^{-d +1},
\end{equation}
then, for all $a(x) \in \KK[x]$ and $b(y) \in \KK[y]$, the map given by
\begin{equation}\label{disc.deriv.d}
\partial(x) = h^d\, b(y), \qquad \partial(y) = h^d\, a(x),
\end{equation}
extends to a skew derivation $(\partial, \sigma_\mu)$ of $A$.
\end{zlist}

The (combinations of the) above maps exhaust all $\sigma_\mu$-skew derivations of  $A$ contained in Theorem~\ref{thm.Weyl.deriv}.
\end{proposition}
\begin{proof}
We study all possible $\sigma_\mu$-skew derivations of  $A$ that satisfy assumptions of  Theorem~\ref{thm.Weyl.deriv}. Since in our case $\sigma$ is the identity map, we first determine $\varphi^n$-skew derivations of the polynomial algebra $\KK[h]$. The action of $\varphi^n$ on any element of $\KK[h]$ results in rescaling the $h$ by $q^n$. Thus any $\varphi^n$-skew derivation $\partial_n$ of $\KK[h]$ takes the form of the multiple of an appropriate Jackson's derivative (understood as the ordinary derivative in case $n=0$),
$$
\partial_n (f(h)) = a_n(h) f'_{q^n}(h) :=  a_n(h) \frac{f\left(q^n\, h\right) - f(h)}{\left(q^n  -1\right)h}.
$$
Requesting that $\partial_n \circ \varphi = \mu\, \varphi \circ\partial_n$, and evaluating it at $h$, yields the constraint
\begin{equation}\label{an}
qa_n(h) = \mu\, a_n(qh),
\end{equation}
which has the following solutions: either 
\begin{rlist}
\item $a_n(h) =0$ and there are no restrictions on $\mu$, or else 
\item
there exists $d\in \NN$ such that $\mu = q^{-d+1}$ (see \eqref{mu.q}) and then $a_n(h)$ is a scalar multiple of $h^d$.
\end{rlist}
These skew derivations provide us with only choices  of maps $\alpha_i$ in Theorem~\ref{thm.Weyl.deriv}. We now look at the derivations listed in equations \eqref{monomials} in the proof of Theorem~\ref{thm.Weyl.deriv}. In the case (i), all derivations $\partial_m$ and $\partial_n$ are trivial, and we are thus left with $\partial_0$,
$$
\partial_0(x) = f(h) x, \qquad \partial_0(y) = -\mu\, \varphi^{-1}(f(h))y = -\mu\, f(q^{-1}h)y.
$$
This proves the first part of statement (1). In the case (ii) we obtain
$$
\partial_m(x) = 0,\quad  \partial_m(y) \sim  h^d x^{m-1}, \qquad  \partial_{-n}(x)  \sim h^d y^{n-1},\quad  \partial_{-n}(y) = 0.
$$
Combining these solutions we obtain  statement (2). 

Clearly, derivations \eqref{disc.deriv.zero} have the property $\partial(h) =0$.  To prove their uniqueness 
we make a general ansatz 
$$
\partial(x) = \sum _{m=0}^M \sum_{n=0}^N \alpha_{m\, n} y^m x^n, 
\qquad 
\partial(y) = \sum _{m=0}^M \sum_{n=0}^N \beta_{m\, n} y^m x^n, 
$$
and then consider the cases $\KK_q[x,y]$ and $D_q(x,y)$ separately. In the former case $\partial(h) = \partial(yx)=0$ implies
$$
\mu^{-1}\sum _{m=0}^M \sum_{n=0}^N \beta_{m\, n} y^m x^{n+1} +  \sum _{m=0}^M \sum_{n=0}^N \alpha_{m\, n} y^{m+1} x^n =0,
$$
i.e.\
\begin{equation}\label{q.plane.yx}
\beta_{m\, n-1} = -\mu\, a_{m-1\, n}.
\end{equation}
On the other hand $q\partial(h) = \partial(xy) = 0$ implies
$$
\mu \sum _{m=0}^M \sum_{n=0}^N q^n\alpha_{m\, n} y^{m+1} x^n + \sum _{m=0}^M \sum_{n=0}^N \beta_{m\, n}q^m y^m x^{n+1}   =0,
$$
i.e.,
\begin{equation}\label{q.plane.xy}
\beta_{m\, n-1} = -\mu\, q^{n-m} a_{m-1\, n}.
\end{equation}
Putting \eqref{q.plane.yx} and \eqref{q.plane.xy} together we thus conclude that the only non-zero coefficients  $\alpha_{m\, n}$ are these of the form $\alpha_{m-1\, m}$, and thus we are lead to the solution \eqref{disc.deriv.zero}.

In the case of the quantum disc algebra, the vanishing of $\partial$ on $yx$ implies \eqref{q.plane.yx}. On the other hand $\partial(xy) =0$ implies
\begin{eqnarray*}
0 &=& \mu \sum _{m=0}^M \sum_{n=0}^N \alpha_{m\, n} y^{m} x^n y+ \sum _{m=0}^M \sum_{n=0}^N \beta_{m\, n} x y^m x^{n}\\
&=& \mu \left(  \sum _{m=0}^M \sum_{n=0}^N \alpha_{m\, n} y^{m} x^n y- \sum _{m=0}^M \sum_{n=0}^N \alpha_{m-1\, n+1} x y^m x^{n}\right),
\end{eqnarray*}
where \eqref{q.plane.yx} has been used. At this point we can use Lemma~\ref{lemma.disc} and compare coefficients at monomials $y^mx^n$ to obtain
\begin{equation}\label{disc.xy}
\left(q^n - q^m\right) \left(\alpha_{m-1\, n}  - q\, \alpha_{m\, n+1}\right) =0.
\end{equation}
For $n=m$  this equation is obviously satisfied with no constraint on $\alpha_{n\, n+1}$. However, when $m\neq n$, equation \eqref{disc.xy} is equivalent to
\begin{equation}\label{disc.xy.2}
\alpha_{m-1\, n}  = q\, \alpha_{m\, n+1}.
\end{equation}
Bearing in mind that indices $m$ and $n$ have finite ranges, the only solution to \eqref{disc.xy.2} is  $\alpha_{m-1\, n} =0$, for all $m\neq n$. As was the case with the quantum plane algebra, we are left with a skew derivation of the type listed in assertion (1).

The final assertion follows from the necessity of solutions to constraints arising from assumptions of Theorem~\ref{thm.Weyl.deriv}.
\end{proof}

\begin{example}\label{example.disc.plane.ortho} In this example we apply Proposition~\ref{prop.ortho} to discuss orthogonal pairs of skew derivations on the quantum disc and polynomial algebras. Note that Proposition~\ref{prop.ortho} is applicable only to elementary skew derivations leading to  derivations of the type \eqref{disc.deriv.d}. Since in this case the skew derivations $\alpha_i$ on $\KK[h]$ evaluated at $h$ (in the case of the quantum plane) or $1-h$ (in the case of the disc) are proportional to $h^d$, the co-primeness requirements of Proposition~\ref{prop.ortho} immediately imply that $d=0$, and hence $\mu=q$.
Hence the elementary skew derivations $\partial_1$, $\partial_2$ can be given by 
\begin{subequations}\label{partials.disc.plane.ortho}
\begin{equation}\label{partial.disc.plane.ortho}
\partial(x) = 0, \qquad \partial (y) = c \, x^{m},
\end{equation}
\begin{equation}\label{bar.partial.disc.plane.ortho}
\bar\partial(x) = \bar{c} y^{n}, \qquad \bar\partial(y) = 0,
\end{equation}
\end{subequations}
for all   $m,n\in \NN$, and  non-zero elements $c,\bar{c}$ of $\KK$. Henceforth we need to consider the quantum plane and quantum polynomial ring cases separately. 
\begin{rlist}
\item In the $\KK_q[x,y]$-case, $a=h$, hence it is never coprime with $\varphi^i(a) = q^i h$, and thus only $m=n=0$ in \eqref{partials.disc.plane.ortho} gives an orthogonal pair. 
\item In the $D_q(x,y)$-case, $a= 1-h$, hence $\varphi^i(a) = 1-q^i h$ is always coprime with $a$ as long as $q$ is not a root of unity (which is assumed in Proposition~\ref{prop.disc.one} and hence in this example). Thus $(\partial, \sigma_q)$, $(\bar\partial, \sigma_q)$ form an orthogonal pair for any choice of $m$ and $n$.
\end{rlist}
\end{example}

The following proposition lists all $\sigma_\mu$- skew derivations of $D_q(x,y)$ and $\KK_q[x,y]$ in the case $\mu =q$.

\begin{proposition}\label{prop.disc.mu=q}
Assume that a non-zero $q\in \KK$ is not a root of unity, and let $A$ be either
\begin{blist}
\item the disc algebra $D_q(x,y)$ or
\item the quantum polynomial ring $\KK_q[x,y]$.
\end{blist}
For all natural numbers $M, N$ and $\alpha_{m\, n}  \in \KK$, $m=0,\ldots M-1$, $n=1,\ldots N$, and 
polynomials $f(x) \in \KK[x]$, $g(y)\in \KK[y]$, 
 the maps defined by
\begin{subequations}\label{disc.full}
\begin{equation}\label{disc.full.x}
\partial(x) = g(y) +  \sum _{m=0}^{M-1} \sum_{n=1}^N \alpha_{m\, n} y^m x^n, 
\end{equation}
\begin{equation}\label{disc.full.y}
\partial(y) =f(x) - q \sum _{m=1}^M \sum_{n=0}^{N-1} \frac{[n+1]_q}{[m]_q} \alpha_{m-1\, n+1}\, y^m x^n,
\end{equation}
\end{subequations}
where the integer in square brackets denotes the $q$-integer
$$
[m]_q = \frac{q^m -1}{q-1},
$$ 
extend to a skew derivation $(\partial, \sigma_q)$ of $A$.

Any $\sigma_q$-skew derivation of $A$ is  of the form \eqref{disc.full}.
\end{proposition}
\begin{proof}
First we consider the case $A=D_q(x,y)$. To prove that $(\partial, \sigma_q)$ extends to a skew derivation of $D_q(x,y)$, we only need to make sure that the definition of $\partial$, while being extended to the whole of  $D_q(x,y)$ remains compatible with the relation \eqref{q.disc}. Note that the twisted Leibniz rule implies that $\partial(q-1) = 0$, hence this amounts to checking
\begin{equation}\label{qxy}
 \partial(xy) = q\partial(yx).
\end{equation}

We start with a general ansatz 
\begin{equation}\label{ansatz}
\partial(x) = \sum _{m=0}^M \sum_{n=0}^N \alpha_{m\, n} y^m x^n, 
\qquad 
\partial(y) = \sum _{m=0}^M \sum_{n=0}^N \beta_{m\, n} y^m x^n, 
\end{equation}
and, with the help of Lemma~\ref{lemma.disc}, compute,
\begin{eqnarray*}
\partial(xy)&=& q\partial(x) y + x\partial(y)  = q\sum _{m=0}^M \sum_{n=0}^N \alpha_{m\, n} y^m x^n  y + x (\sum _{m=0}^M \sum_{n=0}^N \beta_{m\, n} y^m x^n)\\
&=&  \sum _{m=0}^M \sum_{n=0}^N \alpha_{m\, n} q^{n+1}y^{m+1}x^{n}+ q \sum _{m=0}^M \sum_{n=0}^N \alpha_{m\, n}  (1-q^{n})y^{m}x^{n-1}\\
&& + \sum _{m=0}^M \sum_{n=0}^N \beta_{m\, n} q^{m}y^{m}x^{n+1}+\sum _{m=0}^M \sum_{n=0}^N \beta_{m\, n} (1-q^{m})y^{m-1}x^{n}.
\end{eqnarray*}
On the other hand, 
\begin{eqnarray*}
q\partial(yx)&=&\partial(y) x + qy\partial(x) = (\sum _{m=0}^M \sum_{n=0}^N \beta_{m\, n} y^m x^n)x+ q y (\sum _{m=0}^M \sum_{n=0}^N \alpha_{m\, n} y^m x^n)\\
&=& \sum _{m=0}^M \sum_{n=0}^N \beta_{m\, n}y^m x^{n+1}+q\sum _{m=0}^M \sum_{n=0}^N \alpha_{m\, n} y^{m+1} x^{n}.
\end{eqnarray*}
Thus \eqref{qxy} is equivalent to:
\begin{eqnarray*}
\sum _{m=0}^M \sum_{n=0}^N \alpha_{m\, n}(q^{n+1}-q) y^{m+1}x^{n}+\sum _{m=0}^M \sum_{n=0}^N \alpha_{m\, n} q(1-q^{n})y^{m}x^{n-1}\\
+\sum _{m=0}^M \sum_{n=0}^N \beta_{m\, n}(q^{m}-1)y^{m}x^{n+1}
+\sum _{m=0}^M \sum_{n=0}^N \beta_{m\, n}(1-q^{m})y^{m-1}x^{n}=0
\end{eqnarray*}
Looking at the highest powers of $x$ and $y$, we thus conclude that
\begin{equation}\label{abMN}
\beta_{m\, N} =  \alpha_{M n} = 0, \qquad \mbox{for all $m,n \neq 0$}.
\end{equation}
Comparing the coefficients of  $y^{m}x^{n}$, with understanding that $\alpha_{m\, n} = \beta_{m\, n} =0$ whenever $m\not\in \{0,\ldots, M\}$ or $n\not\in \{0,\ldots , N\}$, we obtain the following constraints for the coefficients of $\partial(x)$ and $\partial(y)$,
\begin{equation}\label{ab.constraint}
\beta_{m\,n-1}(q^{m}-1)+\beta_{m+1\,n}(1-q^{m+1})= q\left(\alpha_{m-1\,n}(1 -q^{n})+\alpha_{m\,n+1}(q^{n+1}-1)\right).  
\end{equation}
These have the general solution, 
for all $n = 0,\ldots, N-1$, $m=0,\ldots  ,M-1$,
\begin{equation}\label{ab.gen}
\beta_{m+1\,n}= -q \dfrac{q^{n+1}-1}{q^{m+1}-1}\alpha_{m\,n+1} = -q\frac{[n+1]_q}{[m+1]_q}\alpha_{m\,n+1}.
\end{equation}
Therefore, by combining solutions \eqref{abMN} and \eqref{ab.gen}  we obtain that necessarily
$$
\partial(x) = \sum_{m=0}^M\alpha_{m\, 0}\, y^m +  \sum _{m=0}^{M-1} \sum_{n=1}^N \alpha_{m\, n} y^m x^n, 
$$
and
$$
\partial(y) = \sum_{n=0}^N\beta_{0\, n} \, x^n - q \sum _{m=1}^M \sum_{n=0}^{N-1} \frac{[n+1]_q}{[m]_q} \alpha_{m-1\, n+1}\, y^m x^n. 
$$
Since the compatibility of the $\sigma_q$-skew derivation property of $\partial$ with the relation \eqref{q.disc} is checked, $(\partial, \sigma_q)$ extends to a skew derivation of $D_q(x,y)$.

Also in the case $A=\KK_q[x,y]$ we make the ansatz \eqref{ansatz} and  derive the conditions that coefficients $\alpha_{m\, n}$, $\beta_{m\, n}$ must satisfy in order for \eqref{qxy} to be fulfilled. Using the quantum plane relation \eqref{q.plane} one easily finds that  \eqref{qxy} is equivalent to
\begin{equation}\label{constraint.q.plane}
\sum_{m=0}^M\sum_{n=0}^N (1-q^m)\beta_{m\, n} y^mx^{n+1} = q\sum_{m=0}^M\sum_{n=0}^N (q^n-1)\alpha_{m\, n} y^{m+1}x^{n}. 
\end{equation}
Studying the terms at $y^{M+1}$ and $x^{N+1}$ one concludes that $\beta_{m\, N} =  \alpha_{M\, n} = 0$, unless $m=0$ or $n=0$, and then comparing coefficients at $y^mx^n$ one immediately derives \eqref{ab.gen}. This leads to the stated assertion and thus completes the proof.
\end{proof}

\begin{remark}\label{rem.disc.full}
The skew derivation \eqref{disc.deriv.zero} is obtained from \eqref{disc.full} by setting $\alpha = \beta =0$, $N=M+1$ and $\alpha_{m\, n} =0$ if $n \neq m+1$.  

The skew derivation \eqref{disc.deriv.d} is obtained from \eqref{disc.full} by setting  $\alpha_{m\, n}=0$. 
\end{remark}

Finally, we construct orthogonal pairs of skew derivations on the quantum disc algebra from those listed in Proposition~\ref{prop.disc.mu=q} and not already included in Example~\ref{example.disc.plane.ortho}.

\begin{proposition}\label{prop.ortho.disc}
Assume that a non-zero $q\in \KK$ is not a root of unity. For any  $n,m\in \NN$ such that $m,n>{1}$ and for any non-zero $c,\bar{c} \in \KK$,
consider the following pair of $\sigma_q$-skew derivations on $D_q(x,y)$, defined on the generators by
\begin{subequations}\label{partials.disc.ortho}
\begin{equation}\label{partial.disc.ortho}
\partial(x) = c x^n, \qquad \partial (y) = -q[n]_qc \, (1-h)x^{n-2}.
\end{equation}
\begin{equation}\label{bar.partial.disc.ortho}
\bar\partial(x) = -q^{-1}[m]_q\bar{c} (1-q^{-m+2}h)y^{m-2}, \qquad \bar\partial(y) = \bar{c}  y^m.
\end{equation}
\end{subequations}
For all $k,l\in \NN$ define
$$
q_{kl}=\dfrac{1-[k]_q[l]_q q^{-k+1}}{1-[k]_q[l]_q}.
$$ 
If
\begin{subequations}\label{kl.conditions}
 \begin{equation}\label{kl.condition.1}
q_{kl} \neq{q^i},\qquad  i\in \{-2k+3,-2k+4,\ldots ,-k+1,l,l+1,\ldots ,2l-3, 2l-1\}
\end{equation}
and
\begin{equation}\label{kl.condition.2}
 q_{kl} \neq{q^{2l-2}} q_{lk},
\end{equation}
\end{subequations}
for $(k,l)=(m,n)$ and $(k,l)= (n,m)$,
then $\partial$ and $\bar\partial$ form an orthogonal pair.
\end{proposition}
\begin{proof}
  We will construct sets of elements of $D_q(x,y)$ which will satisfy orthogonality conditions \eqref{ortho.x} for skew derivations \eqref{partials.disc.ortho}. Observe that
\begin{eqnarray*}
\frac{q}{[m]_q} y^n \bar\partial(x) \!\!\!&+&\!\!\! (1-q^{-m-n+2}h)y^{n-2}\bar\partial(y) \\
&=& -\bar{c} y^n (1-q^{-m+2}h)y^{m-2} + \bar{c} (1-q^{-m-n+2}h)y^{m+n-2} = 0,
\end{eqnarray*}
and
\begin{eqnarray*}
\frac{q}{[m]_q} \bar\partial(x) y^n  \!\!\!&+&\!\!\! \bar\partial(y)(1-q^{2}h)y^{n-2} \\
&=& -\bar{c}  (1-q^{-m+2}h)y^{m+n-2} + \bar{c} y^m(1-q^{-2}h) y^{n-2}= 0.
\end{eqnarray*}
On the other hand
\begin{eqnarray} \label{ortho.1}
\frac{q}{[m]_q} \!\!\!\!\!\!&&\!\!\! \!\!\! y^n \partial(x) +(1-q^{-m-n+2}h)y^{n-2}\partial(y) \nonumber \\
&=& \frac{q}{[m]_q} {c} y^n x^n -q[n]_qc \, (1-q^{-m-n+2}h)y^{n-2}(1-h)x^{n-2} \nonumber \\
&=&   \frac{q}{[m]_q} {c} \prod_{k=0}^{n-2}\left(1-q^{-k}h\right)\left(1 - q^{-n+1}h -[m]_q[n]_q\left(1-q^{-m-n+2}h\right)\right) \nonumber\\
&=&qc \left(\frac{1}{[m]_q} -[n]_q\right) \prod_{k=0}^{n-2}\left(1-q^{-k}h\right)\left(1 - q^{-n+1}\frac{1- [m]_q[n]_qq^{-m+1}}{1- [m]_q[n]_q}h\right)\nonumber \\
&=&qc \left(\frac{1}{[m]_q} -[n]_q\right) \prod_{k=0}^{n-2}\left(1-q^{-k}h\right)\left(1 - q^{-n+1}q_{mn} h\right),
\end{eqnarray}
and
\begin{eqnarray}\label{ortho.2}
\frac{q}{[m]_q} \!\!\!\!\!\!&&\!\!\! \!\!\!\partial(x) y^n   + \partial(y)(1-q^{2}h)y^{n-2} \nonumber \\
&=& \frac{q}{[m]_q} {c} x^n y^n -q[n]_qc \, (1-h)x^{n-2}(1-q^{2}h)y^{n-2} \nonumber\\
&=& \frac{q}{[m]_q} {c} (1-q^{n}h) \prod_{k=1}^{n-2}\left(1-q^{k}h\right)\left(1 - q^{n-1}h -[m]_q[n]_q\left(1-h\right)\right) \nonumber\\
&=&qc \left(\frac{1}{[m]_q} -[n]_q\right)(1-q^{n}h)\prod_{k=1}^{n-2}\left(1-q^{k}h\right) \nonumber \\
&& \left(1 - q^{n-1}\frac{1- [m]_q[n]_qq^{-n+1}}{1- [m]_q[n]_q}h\right)\nonumber \\
&=& qc \left(\frac{1}{[m]_q} -[n]_q\right)(1-q^{n}h)\prod_{k=1}^{n-2}\left(1-q^{k}h\right) \left(1 - q^{n-1}q_{nm}h\right) .
\end{eqnarray}
By \eqref{kl.conditions}
$q_{mn}\neq q^{i}$, for all $i\in\{ n,\ldots,2n-3,2n-1\}$,  $q_{nm}\neq q^{i}$, for all $i\in\{-2n+3,-2n+4,\ldots,-n+1\}$ 
and  $q_{mn}\neq q^{2n-2} q_{nm}$. Combining this with the fact that $q$ is not a root of unity,   we conclude that
polynomials in $h$ that appear in \eqref{ortho.1} and \eqref{ortho.2} have no roots in common. Therefore a polynomial combination of them can be found giving 1, and thus the first set of elements that satisfy \eqref{ortho.x} can be constructed.

In a similar way,
\begin{eqnarray*}
(1-q^{m}h)x^{m-2}\partial(x) \!\!\!&+&\!\!\! \frac{q^{-1}}{[n]_q} x^m \partial(y) \\
&=& c (1-q^{m}h)x^{m+n-2} -c x^m (1-h)x^{n-2}= 0,
\end{eqnarray*}
and
\begin{eqnarray*}
\partial(x)(1-q^{-n}h)x^{m-2}  \!\!\!&+&\!\!\! \frac{q^{-1}}{[n]_q} \partial(y) x^{m} \\
&=& c x^n (1-q^{-n}h)x^{m-2} -c(1-h) x^{n+m-2}=0.
\end{eqnarray*}
On the other hand
\begin{eqnarray}\label{ortho.3}
(1-q^{m}h)x^{m-2}\bar\partial(x)\!\!\! &+&\!\!\!  \frac{q^{-1}}{[n]_q}  x^m \bar\partial(y) \nonumber \\
&\hspace{-3.5cm}=&\hspace{-2 cm} -q^{-1}[m]_q\bar{c} \, (1-q^{m}h)x^{m-2}(1-q^{-m+2}h)y^{m-2}+ \frac{q^{-1}}{[n]_q} \bar{c} x^m y^m \nonumber\\
&\hspace{-3.5cm}=&\hspace{-2 cm} \frac{q^{-1}}{[n]_q} \bar{c} (1-q^{m}h) \prod_{k=1}^{m-2}\left(1-q^{k}h\right)\left(-[m]_q[n]_q\left(1-h\right)+1 - q^{m-1}h\right) \nonumber\\
&\hspace{-3.5cm}=&\hspace{-2 cm} q^{-1}\bar{c} \left(\frac{1}{[n]_q} -[m]_q\right)(1-q^{m}h)\prod_{k=1}^{m-2}\left(1-q^{k}h\right) \nonumber \\
 &&\times\left(1 - q^{m-1}\frac{1- [m]_q[n]_qq^{-m+1}}{1- [m]_q[n]_q}h\right)\nonumber \\
&\hspace{-3.5cm}=&\hspace{-2 cm} q^{-1}\bar{c} \left(\frac{1}{[n]_q} -[m]_q\right)(1-q^{m}h)\prod_{k=1}^{m-2}\left(1-q^{k}h\right)\left(1 - q^{m-1}q_{mn}h\right),
\end{eqnarray}
and
\begin{eqnarray}\label{ortho.4}
\bar\partial(x)(1-q^{-n}h)x^{m-2}  \!\!\!&+&\!\!\! \frac{q^{-1}}{[n]_q}\bar\partial(y) x^{m} \nonumber \\
&\hspace{-2.5cm}=&\hspace{-1.5 cm}-q^{-1}[m]_q \bar{c} \, (1-q^{-m+2}h)y^{m-2}(1-q^{-n}h)x^{m-2} +\frac{q^{-1}}{[n]_q} \bar{c} y^m x^m \nonumber \\
&\hspace{-2.5cm}=&\hspace{-1.5 cm} \frac{q^{-1}}{[n]_q} \bar{c} \prod_{k=0}^{m-2}\left(1-q^{-k}h\right)\left(-[m]_q[n]_q\left(1-q^{-n-m+2}h\right)+1 - q^{-m+1}h \right) \nonumber\\
&\hspace{-2.5cm}=&\hspace{-1.5 cm} q^{-1} \bar{c} \left(\frac{1}{[n]_q} -[m]_q\right) \prod_{k=0}^{m-2}\left(1-q^{-k}h\right)
\left(1 - q^{-m+1}\frac{1- [m]_q[n]_qq^{-n+1}}{1- [m]_q[n]_q}h\right)\nonumber \\
&\hspace{-2.5cm}=&\hspace{-1.5 cm} q^{-1} \bar{c} \left(\frac{1}{[n]_q} -[m]_q\right) \prod_{k=0}^{m-2}\left(1-q^{-k}h\right)
\left(1 - q^{-m+1}q_{nm}h\right).
\end{eqnarray}
As before, since $q$ is not a root of unity and by \eqref{kl.conditions}, polynomials in $h$ that appear in \eqref{ortho.3} and \eqref{ortho.4} have no roots in common. Thus a polynomial combination of them can be found giving 1 and then the second set of elements that satisfy \eqref{ortho.x} can be constructed. This proves the statement. 
\end{proof}

\begin{corollary}\label{cor.ortho.disc}
Let $\KK =\CC$ and assume that $q\neq 1$ is a positive  real number. Then, for any  $n\in \NN\setminus\{0,1\}$ and for any non-zero $c,\bar{c} \in \CC$,
the following pair of $\sigma_q$-skew derivations on $D_q(x,y)$, defined on the generators by
\begin{subequations}\label{partials.disc.ortho.ex}
\begin{equation}\label{partial.disc.ortho.ex}
\partial(x) = c x^n, \qquad \partial (y) = -q[n]_qc \, (1-h)x^{n-2}.
\end{equation}
\begin{equation}\label{bar.partial.disc.ortho.ex}
\bar\partial(x) = -q^{-1}[n]_q\bar{c} (1-q^{-n+2}h)y^{n-2}, \qquad \bar\partial(y) = \bar{c}  y^n.
\end{equation}
\end{subequations}
is orthogonal.
\end{corollary}
\begin{proof}
In this case, the conditions \eqref{kl.conditions} reduce to 
$$
q_{nn} \neq{q^i},\qquad  i\in \{-2n+3,-2n+4,\ldots ,-n+1,n,n+1,\ldots ,2n-3, 2n-1\}.
$$
By using the definitions of $[n]_q$ and $q_{nn}$, one easily finds that
$$
q_{nn} = q^{-n} \frac{[n+1]_q}{[n]_q +1}.
$$
Assuming that $q\neq 1$, for $i\leq -n$, the critical equation
\begin{equation}\label{qnn}
q^{-n} \frac{[n+1]_q}{[n]_q +1} = q^i,
\end{equation}
is equivalent to 
$$
[-i-n+1]_q[n]_q +[-i-n]_q =0.
$$
If $i\leq -n$, and $q$ is positive, all the $q$-numbers on the left hand side are positive, and thus there are no solutions. For $i=-n+1$, \eqref{qnn} is equivalent to $(q-1)^2 = 0$, and thus has no  solutions for $q$ other than $1$. Finally, for positive $i$ and $q\neq 1$ the equation \eqref{qnn} is equivalent to
$$
q[n]_q [n+i-1]_q +[n+i]_q =0,
$$ 
which again has no positive solutions for $q$.
\end{proof}

\end{document}